\documentclass[12pt]{article}
\usepackage{amsfonts,amssymb,amsmath,amsthm}
\usepackage{graphics}

\newcommand{\R}{\mathbb{R}}
\newcommand{\N}{\mathbb{N}}
\newcommand{\Z}{\mathbb{Z}}

\newcommand{\D}{\mathcal{D}}
\newcommand{\T}{\mathbb{T}}
\newcommand{\esslim}{\operatornamewithlimits{ess\,lim}}

\newcommand{\essinf}{\operatornamewithlimits{ess\,inf}}
\newcommand{\esssup}{\operatornamewithlimits{ess\,sup}}
\newcommand{\sign}{\operatorname{sign}}
\newcommand{\meas}{\operatorname{meas}}
\newcommand{\Co}{\operatorname{co}}

\newcommand{\MV}{\operatorname{MV}}
\newcommand{\const}{\mathrm{const}}
\renewcommand{\div}{\operatorname{div}}
\newcommand{\supp}{\operatorname{supp}}

\newtheorem{theorem}{Theorem}[section]
\newtheorem{lemma}{Lemma}[section]
\newtheorem{corollary}{Corollary}[section]
\newtheorem{proposition}{Proposition}[section]
\theoremstyle{definition}
\newtheorem{definition}{Definition}[section]
\theoremstyle{remark}
\newtheorem{remark}{Remark}[section]
\newtheorem{example}{Example}[section]
\date{}
\numberwithin{equation}{section}

\title{On entropy solutions of scalar conservation laws with discontinuous flux}

\author{Evgeny~Yu.~Panov \\
Yaroslav-the-Wise Novgorod State University, Russia}

\begin{document}

\maketitle

\begin{abstract}
We introduce the notion of entropy solutions (e.s.) to a conservation law with an arbitrary jump continuous flux vector and prove existence of the largest and the smallest e.s. to the Cauchy problem. The monotonicity and stability properties of these solutions are also established. In the case of a periodic initial function we derive the uniqueness of e.s. Generally, the uniqueness property can be violated, which is confirmed by an example. Finally, we proved
that in the case of single space variable a weak limit of a sequence of spatially periodic e.s. is an e.s. as well.
\end{abstract}

\section{Introduction}
In the half-space $\Pi=\R_+\times\R^n$, where $\R_+=(0,+\infty)$, we consider the conservation law
\begin{equation}\label{1}
u_t+\div_x\varphi(u)=0
\end{equation}
with a jump continuous flux vector $\varphi(u)=(\varphi_1(u),\ldots,\varphi_n(u))$. This means that at each point $u_0\in\R$ there exist one-sided limits
$\displaystyle\lim_{u\to u_0\pm}\varphi(u)\doteq\varphi(u_0\pm)$. For example, if the components $\varphi_i(u)$, $i=1,\ldots,n$, are BV-functions then the vector $\varphi(u)$ is jump-continuous. It is known that the set
$$
D=\{ \ u_0\in\R \ | \ |\varphi(u_0+)-\varphi(u_0)|+|\varphi(u_0)-\varphi(u_0-)|>0 \ \}
$$
of discontinuity points of the vector $\varphi(u)$ is at most countable (and may be an arbitrary at most countable set in $\R$). We used above and will use in the sequel the notation $|\cdot|$ for Euclidean finite-dimensional norms (including the absolute value in one-dimensional case). We will treat $\varphi(u)$ as a multi-valued vector function with values
$\bar\varphi(u_0)=[\varphi(u_0-),\varphi(u_0)]\cup [\varphi(u_0),\varphi(u_0+)]$ being a union of two segments in $\R^n$.
(one may use even more general continuous curves connecting $\varphi(u_0-)$ with $\varphi(u_0+)$ and passing through $\varphi(u_0)$).
Clearly, these sets are different from a single points only if $u_0\in D$. Let us demonstrate that the graph of $\bar\varphi(u)$ admits a continuous parametrization
\begin{equation}\label{cpar}
u=b(v)\in C(\R), \quad \bar\varphi(u)\ni g(v)\in C(\R,\R^n),
\end{equation}
such that the function $b(v)$ is non-strictly increasing and coercive, i.e., $b(v)\to\pm\infty$ as $v\to\pm\infty$.
This was shown in paper \cite{BGMS11}, but only in the case when the set $D$ admits monotone numeration $D=\{u_k\}$, $k\in\N$, $u_{k+1}>u_k$ $\forall k\in\N$, i.e., when $D$ is a completely ordered subset of $\R$. In the following lemma we construct the required parametrization for the general case.

\begin{lemma}\label{lem1}
There exists a parametrization (\ref{cpar}) with a non-strictly increasing and coercive $b(v)$.
\end{lemma}

\begin{proof}
We consider the more complicated case when $D$ is infinite (in the case of finite $D$ we only need to replace the set $\N$ in the proof below by its finite subset).
We numerate set $D$: $D=\{u_k\}_{k\in\N}$ and choose positive numbers $h_k$ such that $\displaystyle\sum_{k=1}^\infty h_k=c<\infty$ (we can take $h_k=2^{-k}$). We define the finite discrete measure $\displaystyle\mu(u)=\sum_{k=1}^\infty h_k\delta(u-u_k)$, where by
$\delta(u-u_k)$ we denote the Dirac mass at the point $u_k$. Then we introduce the strictly increasing function
$\alpha(u)=u+\mu((-\infty,u))$ with jumps at points in $D$. Notice that
\begin{equation}\label{est1}
u\le \alpha(u)\le u+c; \quad \alpha(u_2)-\alpha(u_1)\ge u_2-u_1  \ \forall u_1,u_2\in\R, u_2>u_1.
\end{equation}
The function $b(v)$ is defined as the inverse to the function $\alpha(u)$ considered as maximal monotone graph, that is,
the value $b(v)$ is such $u\in\R$ that $v\in [\alpha(u-),\alpha(u+)]$. It follows from (\ref{est1}) that
$v-c\le b(v)\le v$. If $v_1<v_2$ then denoting $u_i=b(v_i)$, $i=1,2$, we have $v_1\le\alpha(u_1+)\le\alpha(u_2-)\le v_2$ whenever $u_1<u_2$. This relations implies that $v_2-v_1\ge \alpha(u_2-)-\alpha(u_1+)\ge u_2-u_1=b(v_2)-b(v_1)$.
Hence, $b(v_2)-b(v_1)\le v_2-v_1$. In the case $u_1=u_2$ we see that $b(v_2)=b(v_1)$ and the inequality $b(v_2)-b(v_1)\le v_2-v_1$ is evident. The obtained inequality can be written in the form $|b(v_2)-b(v_1)|\le |v_2-v_1|$. We find that $b(v)$ is Lipschitz continuous. Notice also that $b(v)$ takes values $u_k\in D$ on the segments $[a_k,b_k]=[\alpha(u_k-),\alpha(u_k+)]$ of length $h_k>0$. To define the vector $g(v)$, we have to set
$g(v)=\varphi(b(v))$ whenever $b(v)\notin D$. If $b(v)=u_k\Leftrightarrow v\in [a_k,b_k]$ we set $g(v)$ being the piecewise linear function
\begin{equation}\label{ext}
g(v)=\left\{\begin{array}{lcr} \displaystyle\frac{(c_k-v)\varphi(u_k-)+(v-a_k)\varphi(u_k)}{c_k-a_k} & , & a_k\le v\le c_k, \\
\displaystyle\frac{(b_k-v)\varphi(u_k)+(v-c_k)\varphi(u_k+)}{b_k-c_k} & , & c_k\le v\le b_k,
\end{array}\right.
\end{equation}
where $c_k=(a_k+b_k)/2$ (or some other point between $a_k$ and $b_k$).
Let us show that the vector $g(v)$ is continuous on $\R$. We verify that $g(v)$ is continuous at each point $v_0\in\R$.
It is clear if $v_0\in (a_k,b_k)$ for some $k\in\N$, in view of (\ref{ext}). Further, suppose that $v_0\notin [a_k,b_k]$ for all $k\in\N$. This means that
$u_0=b(v_0)\notin D$ and $\varphi(u)$ is continuous at $u_0$. Therefore, for every $\varepsilon>0$ there exists such a $\delta>0$ that in the interval $|u-u_0|<2\delta$ $|\varphi(u)-\varphi(u_0)|<\varepsilon$. This implies
that
\begin{equation}\label{3}
\max(|\varphi(u)-\varphi(u_0)|,|\varphi(u-)-\varphi(u_0)|,|\varphi(u+)-\varphi(u_0)|)\le\varepsilon \quad \forall u\in\R, |u-u_0|<\delta.
\end{equation}
If $|v-v_0|<\delta$ then $|b(v)-u_0|\le |v-v_0|<\delta$ and taking into account (\ref{ext}) and (\ref{3}) we conclude
$$
|g(v)-g(v_0)|\le \max(|\varphi(b(v))-\varphi(u_0)|,|\varphi(b(v)-)-\varphi(u_0)|,|\varphi(b(v)+)-\varphi(u_0)|)\le\varepsilon.
$$
Since $\varepsilon>0$ is arbitrary, this means continuity of $g(v)$ at point $v_0$. By the similar reasons we prove that
$$
\lim_{v\to a_k-} g(v)=\varphi(u_k-)=g(a_k), \ \lim_{v\to b_k+} g(v)=\varphi(u_k+)=g(b_k) \ \forall k\in\N.
$$
Since, in view of (\ref{ext}),
$$
\lim_{v\to a_k+} g(v)=g(a_k), \ \lim_{v\to b_k-} g(v)=g(b_k),
$$
we find that the vector $g(v)$ is continuous at remaining points $v=a_k,b_k$, $k\in\N$. The proof is complete.
\end{proof}

At least formally after the change $u=b(v)$ equation (\ref{1}) reduces to the equation
\begin{equation}\label{1'}
b(v)_t+\div_x g(v)=0
\end{equation}
with already continuous flux $(b(v),g(v))\in\R^{n+1}$.

Recall that entropy solution (e.s.) of equation (\ref{1'}) is a function $v=v(t,x)\in L^\infty(\Pi)$ satisfying
the Kruzhkov entropy condition: $\forall k\in\R$
\begin{equation}\label{ent}
|b(v)-b(k)|_t+\div_x[\sign(v-k)(g(v)-g(k))]\le 0
\end{equation}
in the sense of distributions on $\Pi$ (in $\D'(\Pi))$. This means that for each test function $f=f(t,x)\in C_0^1(\Pi)$, $f\ge 0$
\begin{equation}\label{enti}
\int_\Pi [|b(v)-b(k)|f_t+\sign(v-k)(g(v)-g(k))\cdot\nabla_xf]dtdx\ge 0.
\end{equation}
Taking $k=\pm\|v\|_\infty$, we derive from (\ref{ent}) that $b(v)_t+\div_x g(v)=0$ in $\D'(\Pi)$ and e.s. $v=v(t,x)$ of (\ref{1'}) is a weak solution of this equation.
We study the Cauchy problem for equations (\ref{1}), (\ref{1'}) with initial condition
\begin{equation}\label{ini}
u(0,x)=b(v)(0,x)=u_0(x)\in L^\infty(\R^n).
\end{equation}
This condition is understood in the sense of relation
\begin{equation}\label{ini1}
\esslim_{t\to 0} u(t,\cdot)=u_0 \ \mbox{ in } L^1_{loc}(\R^n).
\end{equation}
It is rather well known (cf. \cite[Proposition~2]{Pan02}) that conditions (\ref{ent}), (\ref{ini1}) can be written in the form of single integral inequality similar to (\ref{enti}): for all $k\in\R$ and each non-negative test function $f=f(t,x)\in C_0^1(\bar\Pi)$, where $\bar\Pi=[0,+\infty)\times\R^n$ being the closure of $\Pi$,
\begin{equation}\label{enti1}
\int_\Pi [|b(v)-b(k)|f_t+\sign(v-k)(g(v)-g(k))\cdot\nabla_xf]dtdx+\int_{\R^n}|u_0(x)-b(k)|f(0,x)dx\ge 0.
\end{equation}
Notice that any jump continuous function is Borel and locally bounded. Therefore, $\varphi(u)\in L^\infty(\Pi)$ for all $u=u(t,x)\in L^\infty(\Pi)$, and we can define the notion of e.s. of original problem (\ref{1}), (\ref{ini}) by the standard Kruzhkov relation like (\ref{enti1})
\begin{equation}\label{enti1a}
\int_\Pi [|u-k|f_t+\sign(u-k)(\varphi(u)-\varphi(k))\cdot\nabla_xf]dtdx+\int_{\R^n}|u_0(x)-k|f(0,x)dx\ge 0
\end{equation}
for all $k\in\R$, $f=f(t,x)\in C_0^1(\bar\Pi)$, $f\ge 0$.
But such e.s. may not exist, see Example~\ref{ex2} below. For the correct definition we need multivalued extension of the flux at discontinuity points and the described above reduction to the well established case of continuous flux.
Apparently, the multivalued extension of the flux was used firstly in \cite{DFR} in the case of some model equation arising in phase transitions.

In the sequel, we need the more general class of measure-valued solutions. Recall
(see \cite{Di,Ta}) that a measure-valued function on $\Pi$ is a weakly measurable map
$(t,x)\to\nu_{t,x}$ of $\Pi$ into the space $\mathrm{Prob}_0(\R)$ of probability Borel measures with
compact support in $\R$.
The weak measurability of $\nu_{t,x}$ means that for each continuous function $p(v)$,
the function $(t,x)\to\langle\nu_{t,x},p(v)\rangle\doteq\int p(v)d\nu_{t,x}(v)$ is Lebesgue-measurable on $\Pi$.
We say that a measure-valued function $\nu_{t,x}$ is bounded if there exists such $R>0$
that $\supp\nu_{t,x}\subset [-R,R]$ for almost all $(t,x)\in\Pi$. We shall denote by $\|\nu_{t,x}\|_\infty$
the smallest of such $R$.
Finally, we say that measure-valued functions of the kind $\nu_{t,x}(v)=\delta(v-v(t,x))$,
where $v(t,x)\in L^\infty(\Pi)$ and $\delta(v-v_*)$ is the Dirac measure at a point $v_*\in\R$, are regular.
We identify these measure-valued functions and the corresponding functions $v(t,x)$,
so that there is a natural embedding $L^\infty(\Pi)\subset \MV(\Pi)$, where by $\MV(\Pi)$ we denote the set of
bounded measure-valued functions on $\Pi$.
Measure-valued functions naturally arise as weak limits of bounded sequences
in $L^\infty(\Pi)$ in the sense of the following theorem by L.~Tartar \cite{Ta}.

\begin{theorem}\label{thT}
Let $v_k(t,x)\in L^\infty(\Pi)$, $k\in\N$, be a bounded sequence. Then there exist a subsequence (we
keep the notation $v_k(t,x)$ for this subsequence) and a bounded measure valued function $\nu_{t,x}\in\MV(\Pi)$
such that
\begin{equation} \label{pr2} \forall p(v)\in C(\R) \quad p(v_k)
\mathop{\to}_{k\to\infty}\langle\nu_{t,x},p(v)\rangle \quad\text{weakly-\/$*$ in } L^\infty(\Pi).
\end{equation}
Besides, $\nu_{t,x}$ is regular, i.e., $\nu_{t,x}(v)=\delta(v-v(t,x))$ if and only if $v_k(t,x)
\mathop{\to}\limits_{k\to\infty} v(t,x)$ in $L^1_{loc}(\Pi)$ (strongly).
\end{theorem}

More generally, the following weak precompactness property holds for bounded sequences of measure valued function, see
for instance \cite[Theorem~2]{Pan95}.

\begin{theorem}\label{thTa}
Let $\nu^k_{t,x}\in MV(\Pi)$, $k\in\N$, be a bounded sequence (this means that the scalar sequence $\|\nu^k_{t,x}\|_\infty$ is bounded). Then there exists a subsequence $\nu^k_{t,x}$ (not relabeled) weakly convergent to a bounded measure valued function $\nu_{t,x}\in\MV(\Pi)$ in the sense of relation
\begin{equation} \label{pr2a} \forall p(v)\in C(\R) \quad \langle\nu^k_{t,x},p(v)\rangle
\mathop{\to}_{k\to\infty}\langle\nu_{t,x},p(v)\rangle \quad\text{weakly-\/$*$ in } L^\infty(\Pi).
\end{equation}
\end{theorem}
Obviously, in the case when the sequence $\nu^k_{t,x}$ consists of regular functions $v_k$, relation (\ref{pr2a}) reduces to (\ref{pr2}). Remark that in Theorems~\ref{thT}, \ref{thTa}  the half-space $\Pi$ may be replaces by arbitrary finite-dimensional domain $\Omega$.

Recall (see \cite{Di,Pan96}) that a measure valued e.s. of (\ref{1'}), (\ref{ini}) is a bounded measure valued function
$\nu_{t,x}\in\MV(\Pi)$, which satisfies the following averaged variant of entropy relation (\ref{enti1}): for all $k\in\R$, $f=f(t,x)\in C_0^1(\bar\Pi)$, $f\ge 0$
\begin{align}\label{entim}
\int_\Pi \left[\int|b(v)-b(k)|d\nu_{t,x}(v)f_t+\int\sign(v-k)(g(v)-g(k))d\nu_{t,x}(v)\cdot\nabla_xf\right]dtdx+ \nonumber\\
\int_{\R^n}|u_0(x)-b(k)|f(0,x)dx\ge 0.
\end{align}
Now we are ready to define the notion of e.s. of original problem (\ref{1}), (\ref{ini}).

\begin{definition}[cf. \cite{BGMS11}]\label{def1}
A function $u=u(t,x)\in L^\infty(\Pi)$ is called an e.s. of problem (\ref{1}), (\ref{ini}) if there exists a measure valued e.s. $\nu_{t,x}(v)$ of (\ref{1'}), (\ref{ini}) such that the push-forward measure $b^*\nu_{t,x}(u)$ coincides with the Dirac mass $\delta(u-u(t,x))$ for a.e. $(t,x)\in\Pi$.
\end{definition}

In view of the requirement $b^*\nu_{t,x}(u)=\delta(u-u(t,x))$ entropy relation (\ref{entim}) can be written as
\begin{align}\label{entim1}
\int_\Pi \left[|u-b(k)|f_t+\int\sign(v-k)(g(v)-g(k))d\nu_{t,x}(v)\cdot\nabla_xf\right]dtdx+ \nonumber\\
\int_{\R^n}|u_0(x)-b(k)|f(0,x)dx\ge 0.
\end{align}

\begin{remark}\label{rem1}
If $u(t,x)$ is an e.s. of (\ref{1}), (\ref{ini}) then $u=-u(t,x)$ is an e.s. of the problem
\begin{equation}\label{1-}
u_t-\div_x\varphi(-u)=0, \quad u(0,x)=-u_0(x)
\end{equation}
regarding to the continuous parametrization $u=-b(-v)$, $-\bar\varphi(-u)=-g(-v)$ of the flux.
In fact, let $\nu_{t,x}$ be a measure valued e.s. of (\ref{1'}), (\ref{ini}) such that $b^*\nu_{t,x}(u)=\delta(u-u(t,x))$.
Then the measure valued function $\tilde\nu_{t,x}=l^*\nu_{t,x}\in\MV(\Pi)$, where $l(v)=-v$, is a measure valued e.s. of the problem (\ref{1-}). In fact, for each $k\in\R$
\begin{align*}
\int |-b(-v)-(-b(-k))|d\tilde\nu_{t,x}(v)=\int |b(v)-b(-k)|d\nu_{t,x}(v)=|u-b(-k)|, \\
\int \sign(v-k)(-g(-v)-(-g(-k)))d\tilde\nu_{t,x}(v)=\int \sign(v+k)(g(v)-g(-k))d\nu_{t,x}(v), \\ |-u_0(x)-(-b(-k))|=|u_0(x)-b(-k)|
\end{align*}
and these equalities imply that for every $f=f(t,x)\in C_0^1(\bar\Pi)$, $f\ge 0$
\begin{align*}
\int_\Pi \left[\int|-b(-v)-(-b(-k))|d\tilde\nu_{t,x}(v)f_t+\right. \\ \left.\int\sign(v-k)(-g(-v)-(-g(-k)))d\tilde\nu_{t,x}(v)\cdot\nabla_xf\right]dtdx+ \\
\int_{\R^n}|-u_0(x)-(-b(-k))|f(0,x)dx=\\
\int_\Pi \left[\int|b(v)-b(-k)|d\nu_{t,x}(v)f_t+\int\sign(v+k)(g(v)-g(-k))d\nu_{t,x}(v)\cdot\nabla_xf\right]dtdx+ \\
\int_{\R^n}|u_0(x)-b(-k)|f(0,x)dx\ge 0
\end{align*}
by entropy relation (\ref{entim}) with $k$ replaced by $-k$. Further,
$$
(-b(-\cdot))^*\tilde\nu_{t,x}(u)=(-b)^*\nu_{t,x}(u)=l^*\delta(u-u(t,x))=\delta(u-(-u(t,x))).
$$
We conclude that $-u(t,x)$ satisfies all the requirement of Definition~\ref{def1} for the problem (\ref{1-}).
\end{remark}

\begin{remark}\label{rem2}
The notion of e.s. does not depend on the choice of parametrization (\ref{cpar}). This follow from the following observation. Let
$$
u=b(v(s)), \quad \bar\varphi(u)\ni g(v(s))
$$
be another parametrization, where $v(s)$ is a continuous increasing and coercive function on $\R$. Then $\tilde\nu_{t,x}(s)$ is a measure valued e.s. of the equation
$$
b(v(s))_t+\div_x g(v(s))=0
$$
if and only if the push-forward measure valued function $\nu_{t,x}=(v^*\tilde\nu_{t,x})(v)$ is a measure valued e.s. of (\ref{1'}).
\end{remark}

In \cite{BGMS11} (also see \cite{BGS13,GSWZ14}) the existence and uniqueness of e.s. were established only in the case
of integrable initial function $u_0\in L^1(\R^n)$ and under assumption of H\"older continuity of the flux vector $\varphi(u)$ at zero with the exponent $\alpha\ge (n-1)/n$.

The main our result is the existence of the largest and the smallest e.s. of (\ref{1}), (\ref{ini}) in the general case $u_0\in L^\infty(\R^n)$. The uniqueness of e.s. follows from this result in the particular case when initial function $u_0$ is periodic. This extends results of \cite{Pan02}. In the case $n=1$ we also prove the weak completeness of the set of spatially periodic e.s., generalizing results of \cite{PanSIMA} to the case of discontinuous flux.

\section{Some properties of e.s.}

We denote $z^\pm=\max(\pm z,0)$, $\sign^+ z=(\sign z)^+$, $\sign^- z=-\sign^+ (-z)$ (so that $\sign^\pm z=\frac{d}{dz}z^\pm$).

\begin{proposition}\label{pro1}
If $u=u(t,x)$ is an e.s. of (\ref{1}), (\ref{ini}), $c\in\R$, then for a.e. $t>0$
$$
\int_{\R^n}(u(t,x)-c)^\pm dx\le\int_{\R^n}(u_0(x)-c)^\pm dx.
$$
\end{proposition}

\begin{proof}
Without loss of generality we will suppose that $\int_{\R^n}(u_0(x)-c)^\pm dx<\infty$, otherwise the required estimate is evident.
It follows from (\ref{entim1}) with $k=\pm M$, $M\ge\|\nu_{t,x}\|_\infty$, that for each $f=f(t,x)\in C_0^1(\bar\Pi)$
\begin{equation}\label{weak}
\int_\Pi \left[uf_t+\int g(v)d\nu_{t,x}(v)\cdot\nabla_xf\right]dtdx+
\int_{\R^n}u_0(x)f(0,x)dx=0.
\end{equation}
Taking into account that for every constant $k\in\R$
$$
\int_\Pi [b(k)f_t+g(k)\cdot\nabla_xf]dtdx+\int_{\R^n}b(k)f(0,x)dx=0,
$$
we can rewrite the previous identity in the form
$$
\int_\Pi \left[(u-b(k))f_t+\int (g(v)-g(k))d\nu_{t,x}(v)\cdot\nabla_xf\right]dtdx+\nonumber\\
\int_{\R^n}(u_0(x)-b(k))f(0,x)dx=0.
$$
Putting this equality together with entropy inequality (\ref{entim1}) and taking into account that
$|z|+z=2z^+$, $\sign z+ 1 =2\sign^+ z$, we arrive at the relation
\begin{align}\label{entim1+}
\int_\Pi \left[(u-b(k))^+f_t+\int\sign^+(v-k)(g(v)-g(k))d\nu_{t,x}(v)\cdot\nabla_xf\right]dtdx+ \nonumber\\
\int_{\R^n}(u_0(x)-b(k))^+f(0,x)dx\ge 0.
\end{align}
By coercivity condition there is such $d\in\R$ that $c=b(d)$.
Let $m\ge n$, $\delta>0$, $\beta(s)=\min((s/\delta)^+,1)^m$.
Integrating the  inequality (\ref{entim1+})
over the measure $\beta'(b(k)-c)db(k)$, we arrive at the relation
\begin{align}\label{5}
\int_\Pi \left[\eta(u-c)f_t+\int q(v)d\nu_{t,x}(v)\cdot\nabla_xf\right]dtdx+
\int_{\R^n}\eta(u_0(x)-c)f(0,x)dx\ge 0,
\end{align}
where
\begin{align*}
\eta(b(v)-c)=\int_d^v (b(v)-b(k))^+\beta'(b(k)-c)db(k)=\int_d^v\beta(b(k)-c)db(k)=\\
\left\{ \begin{array}{lcr} ((b(v)-c)^+)^{m+1}/((m+1)\delta^m) & , & b(v)-c<\delta, \\
b(v)-c-m\delta/(m+1) & , & b(v)-c\ge\delta, \end{array}\right.  \\
q(v)=\int_d^v\sign^+(v-k)(g(v)-g(k))\beta'(b(k)-c)d b(k).
\end{align*}
In particular, if $\supp\nu_{t,x}\subset [-M,M]$ a.e. on $\Pi$, and $\displaystyle C=2\max_{|v|\le M+d}|g(v)|$ then for all $v\in [-M,M]$
$$|q(v)|\le C\int_d^v\beta'(b(k)-c)db(k)=C\beta(b(v)-c),$$ which implies that
\begin{equation}\label{6}
\left|\int q(v) d\nu_{t,x}(v)\right|\le C\int\beta(b(v)-c)d\nu_{t,x}(v)=C\beta(u-c).
\end{equation}
Now we fix $\varepsilon>0$. Since $\beta(s)=1$ for $s>\delta$, the function $\displaystyle\gamma(s)\doteq \frac{\beta(s)}{\eta(s)+\varepsilon}$ decreases on $[\delta,+\infty)$. This implies that
$$
\max\gamma(s)=\max_{s\in[0,\delta]} \gamma(s)\le\max_{s>0}\frac{(s/\delta)^m}{\delta(s/\delta)^{m+1}/(m+1)+\varepsilon}=\max_{\sigma=s/\delta>0}\frac{m+1}{\delta \sigma+(m+1)\varepsilon \sigma^{-m}}.
$$
By direct computations we find
$$
\min_{\sigma>0}(\delta\sigma+(m+1)\varepsilon \sigma^{-m})=\frac{\delta(m+1)}{m}\left(\frac{m(m+1)\varepsilon}{\delta}\right)^{\frac{1}{m+1}}.
$$
Therefore,
$$
\gamma(s)\le\frac{m}{\delta}\left(\frac{\delta}{m(m+1)}\right)^{\frac{1}{m+1}}\varepsilon^{-\frac{1}{m+1}}.
$$
This together with estimate (\ref{6}) implies that
\begin{equation}\label{6a}
\left|\int q(v)d\nu_{t,x}(v)\right|\le N(\eta(u-c)+\varepsilon),
\end{equation}
where
\begin{equation}\label{Ne}
N=N(\varepsilon)=\frac{Cm}{\delta}\left(\frac{\delta}{m(m+1)}\right)^{\frac{1}{m+1}}\varepsilon^{-\frac{1}{m+1}}.
\end{equation}
Since  $\int_\Pi f_tdtdx+\int_{\R^n} f(0,x)dx=0$ we can write (\ref{5}) in the form
\begin{align}\label{7}
\int_\Pi \left[(\eta(u-c)+\varepsilon) f_t+\int q(v)d\nu_{t,x}(v)\cdot\nabla_xf\right]dtdx+\nonumber\\
\int_{\R^n}(\eta(u_0(x)-c)+\varepsilon) f(0,x)dx\ge 0.
\end{align}
Let $E$ be a set of $t>0$ such that $(t,x)$ is a Lebesgue point of $u(t,x)$ for almost all $x\in\R^n$. It is rather well-known (see for example \cite[Lemma~1.2]{PaJHDE}) that $E$ is a set of full measure and $t\in E$ is a common Lebesgue point of the functions $t\to\int_{\R^n} u(t,x)h(x)dx$ for all $h(x)\in L^1(\R^n)$. Since every Lebesgue point of a bounded function $u$ is also a Lebesgue point of $p(u)$ for an arbitrary function $p\in C(\R)$, we may replace $u$ in the above property by $p(u)$, and in particular by $\eta(u-c)+\varepsilon$. We choose a function $\omega(s)\in C_0^\infty(\R)$, such that $\omega(s)\ge 0$, $\supp\omega\subset [0,1]$, $\int\omega(s)ds=1$, and define the sequences $\omega_r(s)=r\omega(rs)$, $\theta_r(s)=\int_{-\infty}^s\omega_r(\sigma)d\sigma=\int_{-\infty}^{rs}\omega(\sigma)d\sigma$, $r\in\N$. Obviously, the sequence $\omega_r(s)$ converges as $r\to\infty$ to the Dirac $\delta$-measure weakly in $\D'(\R)$ while the sequence $\theta_r(s)$ converges to the Heaviside function $\theta(s)$ pointwise and in $L^1_{loc}(\R)$.
Now we take the test function in the form
$$
f=f(t,x)=h\theta_r(t_0-t), \quad h=\rho(N(t-t_0)+|x|-R),
$$
where $\rho(\sigma)\in C^\infty(\R)$ is a decreasing function such that $\rho(\sigma)=1$ for $\sigma\le 0$ and $\rho(\sigma)=0$ for $\sigma\ge 1$
(we can take $\rho(\sigma)=1-\theta_1(\sigma)$), $R>0$, and $t_0\in E$. Observe that $f=\theta_r(t_0-t)$ in a vicinity $|x|<R$ of the singular point $x=0$ and therefore $f\in C^\infty(\bar\Pi)$, $f\ge 0$.
Applying (\ref{7}) to the test function $f$, we arrive at the relation
\begin{align}\label{8}
\int_{\R^n}(\eta(u_0(x)-c)+\varepsilon) h(0,x)dx-\int_\Pi (\eta(u-c)+\varepsilon) h\omega_r(t_0-t)dtdx+\nonumber\\
\int_\Pi \left[N\eta(u_0(x)-c)+\int q(v)d\nu_{t,x}(v)\cdot\frac{x}{|x|}\right]\rho'(N(t-t_0)+|x|-R)\theta_r(t_0-t)dtdx\ge 0
\end{align}
for sufficient large $r\in\N$ such that $rt_0>1$.
In view of (\ref{6a}) and the condition $\rho'(\sigma)\le 0$, the last integral in (\ref{8}) is non-positive and it follows that
$$
\int_\Pi (\eta(u-c)+\varepsilon) h\omega_r(t_0-t)dtdx\le \int_{\R^n}(\eta(u_0(x)-c)+\varepsilon) h(0,x)dx.
$$
Dropping $\varepsilon$ in the left integral, we obtain the inequality
$$
\int_0^\infty \left(\int_{\R^n}\eta(u(t,x)-c)h(t,x)dx\right)\omega_r(t_0-t)dt\le
\int_{\R^n}(\eta(u_0(x)-c)+\varepsilon) h(0,x)dx.
$$
Since $t_0\in E$ is a Lebesgue point of the function $t\to\int_{\R^n}\eta(u(t,x)-c)h(t,x)dx$, we can pass to the limit as
$r\to\infty$ in the above inequality, resulting in
$$
\int_{\R^n}\eta(u(t_0,x)-c)h(t_0,x)dx\le \int_{\R^n}(\eta(u_0(x)-c)+\varepsilon) h(0,x)dx.
$$
Revealing this relation, we get
\begin{align}\label{9}
\int_{\R^n}\eta(u(t_0,x)-c)\rho(|x|-R)dx\le \int_{\R^n}(\eta(u_0(x)-c)+\varepsilon)\rho(|x|-Nt_0-R)dx\le \nonumber \\ \int_{\R^n}\eta(u_0(x)-c)dx+\varepsilon \int_{\R^n}\rho(|x|-Nt_0-R)dx.
\end{align}
With the help of (\ref{Ne}), we obtain that for some constants $c_1$, $c_2=c_2(R,\delta)$
$$
\varepsilon\int_{\R^n}\rho(|x|-N(\varepsilon)t_0-R)dx\le c_1\varepsilon(N(\varepsilon)t_0+R+1)^n\le c_2\varepsilon(1+t_0\varepsilon^{-\frac{1}{m+1}})^n\mathop{\to}_{\varepsilon\to 0+}0
$$
(recall that $m+1>n$). Therefore, passing to the limit in (\ref{9}) as $\varepsilon\to 0+$, we obtain that for all $t_0\in E$
\begin{equation}\label{10}
\int_{\R^n}\eta(u(t_0,x)-c)\rho(|x|-R)dx\le \int_{\R^n}\eta(u_0(x)-c)dx.
\end{equation}
Now observe that $0\le\eta(s)\le s^+$ and $\eta(s)\to s^+$ as $\delta\to 0$. By Lebesgue dominated convergence theorem it follows from (\ref{10}) in the limit as $\delta\to 0$ that for a.e. $t=t_0>0$
$$
\int_{\R^n}(u(t,x)-c)^+\rho(|x|-R)dx\le \int_{\R^n}(u_0(x)-c)^+dx<+\infty.
$$
By Fatou lemma this implies in the limit as $R\to\infty$ that
\begin{equation}\label{11}
\int_{\R^n}(u(t,x)-c)^+dx\le \int_{\R^n}(u_0(x)-c)^+dx,
\end{equation}
as required. In view of Remark~\ref{rem1} the function $-u(t,x)$ is an e.s. of the problem
$u_t-\div_x\varphi(-u)_x-0$, $u(0,x)=-u_0(x)$. Applying (\ref{11}) to this e.s. with $c$ replaced by $-c$, we obtain the inequality
\begin{equation}\label{11a}
\int_{\R^n}(u(t,x)-c)^-dx\le \int_{\R^n}(u_0(x)-c)^-dx  \quad \forall t\in E.
\end{equation}
\end{proof}

\begin{corollary}\label{cor1} Any e.s. $u=u(t,x)$ of (\ref{1}), (\ref{ini}) satisfies the \textbf{maximum/minimum principle}
$$
a=\essinf u_0(x)\le u(t,x)\le b=\esssup u_0(x)  \ \mbox{ for a.e. } (t,x)\in\Pi.
$$
\end{corollary}

\begin{proof}
The maximum/minimum principles directly follows from (\ref{11}) and (\ref{11a}) with $k=b$ and $k=a$, respectively.
\end{proof}
Putting inequalities (\ref{11}), (\ref{11a}) together and using the known relation $|z|=z^++z^-$, we obtain the following

\begin{corollary}\label{cor2} If $u(t,x)$ is an e.s. of (\ref{1}), (\ref{ini}) then for a.e. $t>0$
$$
\int_{\R^n}|u(t,x)-c|dx\le \int_{\R^n}|u_0(x)-c|dx.
$$
\end{corollary}

If $u_1$, $u_2$ is a pair of e.s. and $\nu_{t,x}^{(1)}$, $\nu_{t,x}^{(2)}$ are the corresponding measure valued e.s. of (\ref{1'}) then by a measure-valued analogue of the doubling variable method, developed in \cite{Pan96}, we have the relation
\begin{align*}
\frac{\partial}{\partial t}\iint (b(v)-b(w))^+d\nu_{t,x}^{(1)}(v)d\nu_{t,x}^{(2)}(w)+ \\ \div_x\iint\sign^+(v-w)(g(v)-g(w))d\nu_{t,x}^{(1)}(v)d\nu_{t,x}^{(2)}(w)\le 0 \mbox{ in } \D'(\Pi).
\end{align*}
Since $b(v)=u_1(t,x)$, $b(w)=u_2(t,x)$ on $\supp\nu_{t,x}^{(1)}$, $\supp\nu_{t,x}^{(1)}$, respectively, then the above relation can be written as
\begin{equation}\label{12}
\frac{\partial}{\partial t}(u_1-u_2)^++\div_x\iint\sign^+(v-w)(g(v)-g(w))d\nu_{t,x}^{(1)}(v)d\nu_{t,x}^{(2)}(w)\le 0 \mbox{ in } \D'(\Pi).
\end{equation}

\begin{proposition}\label{pro2}
Let $u_1$, $u_2$ be e.s. of (\ref{1}), (\ref{ini}) with initial functions $u_{10}$, $u_{20}$, respectively. Assume that
for every $T>0$
$$
\meas\{ \ (t,x)\in (0,T)\times\R^n \ | \ u_1(t,x)\ge u_2(t,x) \ \}<+\infty.
$$
Then for a.e. $t>0$
$$
\int_{\R^n} (u_1(t,x)-u_2(t,x))^+dx\le\int_{\R^n} (u_{10}(x)-u_{20}(x))^+dx.
$$
In particular, $u_1(t,x)\le u_2(t,x)$ a.e. in $\Pi$ whenever $u_{10}(x)\le u_{20}(x)$ a.e. in $\R^n$ (the comparison principle).
\end{proposition}

\begin{proof}
Let, as above, $\nu_{t,x}^{(1)}$, $\nu_{t,x}^{(2)}$ be measure valued e.s. of (\ref{1'}) corresponding to $u_1$, $u_2$.
Let $E\subset\R_+$ be a set of full measure similar to one in the proof of Proposition~\ref{pro1} consisting of
values $t>0$ such that $(t,x)$ is a Lebesgue point of $(u_1(t,x)-u_2(t,x))^+$ for a.e. $x\in\R^n$. Then $t\in E$ is a common Lebesgue point of the functions $t\to\int (u_1(t,x)-u_2(t,x))^+ h(x)dx$, $h(x)\in L^1(\R^n)$.
Let $t_0,t_1\in E$, $t_0<t_1$, $\chi_r(t)=\theta_r(t-t_0)-\theta_r(t-t_1)$, where the sequence $\theta_r(t)$, $r\in\N$, was defined in the proof of Proposition~\ref{pro1}.
Applying (\ref{12}) to the nonnegative test function $f(t,x)=\chi_r(t)q(x/R)$, where $q=q(y)\in C_0^1(\R^n)$, $0\le q\le 1$, $q(0)=1$, and $R>0$, we get
\begin{align*}
\int_\Pi (u_1(t,x)-u_2(t,x))^+(\omega_r(t-t_0)-\omega_r(t-t_1))q(x/R)dtdx+ \\
\frac{1}{R}\int_{\Pi} \iint\sign^+(v-w)(g(v)-g(w))d\nu_{t,x}^{(1)}(v)d\nu_{t,x}^{(2)}(w)\cdot\nabla_yq(x/R)\chi_r(t)dtdx\ge 0.
\end{align*}
Since $t_i$, $i=1,2$, are Lebesgue points of the functions $\int_{\R^n} (u_1(t,x)-u_2(t,x))^+q(x/R)dx$ while the sequence
$\chi_r(t)$ is uniformly bounded and converges pointwise to the indicator function of the interval $(t_0,t_1]$, we can pass to the limit as $r\to\infty$ in the above relation and get
\begin{align}\label{13}
\int_{\R^n} (u_1(t_1,x)-u_2(t_1,x))^+q(x/R)dx\le \int_{\R^n} (u_1(t_0,x)-u_2(t_0,x))^+q(x/R)dx+ \nonumber\\
\frac{1}{R}\int_{(t_0,t_1)\times\R^n} \iint\sign^+(v-w)(g(v)-g(w))d\nu_{t,x}^{(1)}(v)d\nu_{t,x}^{(2)}(w)\cdot\nabla_y q(x/R)dtdx.
\end{align}
It follows from the inequality
$$
|(u_1(t_0,x)-u_2(t_0,x))^+-(u_{10}(x)-u_{20}(x))^+|\le |u_1(t_0,x)-u_{10}(x)|+|u_2(t_0,x)-u_{20}(x)|
$$
and initial relation (\ref{ini1}) that
$$
\esslim_{t_0\to 0}(u_1(t_0,x)-u_2(t_0,x))^+=(u_{10}(x)-u_{20}(x))^+ \ \mbox{ in } L^1_{loc}(\R^n).
$$
This allows to pass to the limit as $t_0\to 0$ in (\ref{13}), resulting in the relation: for a.e. $T=t_1>0$
\begin{align}\label{14}
\int_{\R^n} (u_1(T,x)-u_2(T,x))^+q(x/R)dx\le \int_{\R^n} (u_{10}(x)-u_{20}(x))^+q(x/R)dx+ \nonumber\\
\frac{1}{R}\int_{(0,T)\times\R^n} \iint\sign^+(v-w)(g(v)-g(w))d\nu_{t,x}^{(1)}(v)d\nu_{t,x}^{(2)}(w)\cdot\nabla_y q(x/R)dtdx\le \nonumber\\ \int_{\R^n} (u_{10}(x)-u_{20}(x))^+dx+\frac{1}{R}\int_{(0,T)\times\R^n}G(t,x)\cdot\nabla_y q(x/R)dtdx,
\end{align}
where
$$
G=G(t,x)\doteq\iint\sign^+(v-w)(g(v)-g(w))d\nu_{t,x}^{(1)}(v)d\nu_{t,x}^{(2)}(w).
$$
By Definition~\ref{def1}
$b(v)\equiv u_1(t,x)$ on $\supp\nu_{t,x}^{(1)}$, $b(w)\equiv u_2(t,x)$ on $\supp\nu_{t,x}^{(2)}$ and if $u_1(t,x)<u_2(t,x)$ then $v<w$ whenever $v\in \supp\nu_{t,x}^{(1)}$, $w\in\supp\nu_{t,x}^{(2)}$ and therefore the vector-function
$G$ can be different from zero vector only on the set $\{u_1(t,x)\ge u_2(t,x)\}$, which has finite measure in any layer $\Pi_T=(0,T)\times\R^n$. Thus,
denoting $D=\{ \ (t,x)\in\Pi_T \ | \ u_1(t,x)\ge u_2(t,x) \ \}$, we find
\begin{align*}
\left|\int_{(0,T)\times\R^n} G(t,x)\cdot\nabla_y q(x/R)dtdx\right|=\\ \left|\int_D G(t,x)\cdot\nabla_y q(x/R)dtdx\right|\le \|G\|_\infty\|\nabla_y q\|_\infty\meas D<\infty
\end{align*}
(notice that $\displaystyle\|G\|_\infty\le 2\max_{|v|\le M} |g(v)|$, where $M=\max(\|\nu_{t,x}^{(1)}\|_\infty,\|\nu_{t,x}^{(2)}\|_\infty)$).
We see that the last term in (\ref{14}) disappears in the limit as $R\to\infty$ due to the factor $1/R$. Hence, passing to the limit as $R\to\infty$ and using Fatou's lemma (observe that $q(x/R)\mathop{\to}\limits_{R\to\infty} q(0)=1$), we arrive at the desired relation: for all $T\in E$
$$
\int_{\R^n} (u_1(T,x)-u_2(T,x))^+dx\le \int_{\R^n} (u_{10}(x)-u_{20}(x))^+dx.
$$
\end{proof}

The following result asserts the strong completeness of the set of e.s. of the problem (\ref{1}), (\ref{ini}). More precisely, we consider the approximate problem
\begin{equation}\label{1'r}
u_t+\div_x g(v)=0, \ u=b_r(v); \quad  u(0,x)=u_{r0}(x),
\end{equation}
where the sequence $b_r(u)\in C(\R)$, $r\in\N$, of non-strictly increasing functions converges as $r\to\infty$ to $b(u)$ uniformly on any segment.

\begin{proposition}\label{pro3}
Let $u_{r0}=u_{r0}(x)$, $r\in\N$, be a bounded sequence in $L^\infty(\R^n)$, and $u_r=u_r(t,x)$ be a sequence of e.s. of (\ref{1'r}). Assume that as $r\to\infty$ the sequences $u_{r0}\to u_0=u_0(x)$, $u_r\to u=u(t,x)$ in $L^1_{loc}(\R^n)$, $L^1_{loc}(\Pi)$, respectively. Then $u$ is an e.s.  of (\ref{1}), (\ref{ini}) with initial data $u_0$.
\end{proposition}

\begin{proof}
Let $M=\sup\limits_{r\in\N} \|u_{r0}\|_\infty$. By Corollary~\ref{cor1} we see that $\|u_r\|_\infty\le M$ for all $r\in\N$. By Definition~\ref{def1} there exists a sequence $\nu_{t,x}^r\in\MV(\Pi)$ such that
\begin{equation}\label{unu}
b_r^*\nu_{t,x}^r(u)=\delta(u-u_r(t,x)),
\end{equation}
and that for all $k\in\R$ for every $f=f(t,x)\in C_0^1(\bar\Pi)$, $f\ge 0$
\begin{align}\label{enti2}
\int_\Pi \left[|u_r-b_r(k)|f_t+\int\sign(v-k)(g(v)-g(k))d\nu_{t,x}^r(v)\cdot\nabla_xf\right]dtdx+ \nonumber\\
\int_{\R^n}|u_{r0}(x)-b_r(k)|f(0,x)dx\ge 0.
\end{align}
By the coercivity assumption, there exist such a constant $R>0$ that $b(-R)<-M$, $b(R)>M$. Since $b_r(\pm R)\to b(\pm R)$ as $r\to\infty$, we find that $b_r(-R)<-M$, $b_r(R)>M$ for sufficiently large $r$. Without loss of generality we can suppose
that these inequalities holds for all $r\in\N$.
Then, in view of (\ref{unu}), $\supp\nu_{t,x}^r\subset [-R,R]$. Therefore, the sequence of measure valued functions $\nu_{t,x}^r$ is bounded and by Theorem~\ref{thTa} some subsequence of $\nu_{t,x}^r$ converges weakly to a bounded measure valued function $\nu_{t,x}$ (in the sense of relation (\ref{pr2a})). We replace the original sequences $u_{r0}$, $u_r$, $\nu_{t,x}^r$ by the corresponding subsequences (keeping the notations), and pass to the limit as $r\to\infty$ in (\ref{enti2}). As a result, we get
\begin{align}\label{enti3}
\int_\Pi \left[|u-b(k)|f_t+\int\sign(v-k)(g(v)-g(k))d\nu_{t,x}(v)\cdot\nabla_xf\right]dtdx+ \nonumber\\
\int_{\R^n}|u_{0}(x)-b(k)|f(0,x)dx\ge 0
\end{align}
for all $k\in\R$ and each $f=f(t,x)\in C_0^1(\bar\Pi)$, $f\ge 0$.
Moreover, passing to the limit as $r\to\infty$ in the relation (following from (\ref{unu}))
$$
\int q(b_r(v))d\nu_{t,x}^r(v)= q(u_r(t,x)) \quad \forall q(u)\in C(\R),
$$
with the help of the relation $q(b_r(v))-q(b(v))\rightrightarrows 0$ uniformly on $[-R,R]$, we obtain that for a.e. $(t,x)\in\Pi$
\begin{equation}\label{unu1}
\int q(b(v))d\nu_{t,x}(v)= q(u(t,x)).
\end{equation}
A set of full measure $E$ of points $(t,x)$, for which relation (\ref{unu1}) holds can be chosen common for all $q$ from a countable dense subset of $C(\R)$. By the density, this relation remains valid for all $q\in C(\R)$, which evidently means that $b^*\nu_{t,x}(u)=\delta(u-u(t,x))$ for all $(t,x)\in E$. In particular,  it follows from (\ref{enti3}) that
the entropy relation (\ref{entim}) is fulfilled, and $\nu_{t,x}$ is a measure valued e.s. of (\ref{1'}), (\ref{ini}). In correspondence with Definition~\ref{def1}, we conclude that $u$ is an e.s. of (\ref{1}), (\ref{ini}), as required.
\end{proof}

\section{Existence of e.s.}
In this section we assume that the initial function is integrable, $u_0\in L^1(\R^n)\cap L^\infty(\R^n)$. The general case will be treated in the next section, where we will establish existence of the largest and the smallest e.s.

We introduce the approximations $b_r(u)=b(u)+u/r$, $r\in\N$, of $b(u)$ by strictly increasing functions. Then the equation in (\ref{1'r}) can be written in the standard form
\begin{equation}\label{ap1}
u_t+\div_x\varphi_r(u)=0,
\end{equation}
where $\varphi_r(u)=g((b_r)^{-1}(u))\in C(\R,\R^n)$.
As was established in \cite{ABK}, there exists the unique largest e.s. $u_r=u_r(t,x)$ of the Cauchy problem for equation (\ref{ap1}) with initial data $u_0(x)$. It is known that after possible correction on a set of null measure
$u_r(t,\cdot)\in C([0,+\infty),L^1(\R))$. Moreover, for each fixed $r\in\N$ the maps $u_0\to u_r(t,\cdot)$, $t\ge 0$, are nonexpansive in $L^1(\R^n)$. It is clear that for every $\Delta x\in\R^n$ the shifted functions $u_r(t,x+\Delta x)$ are the largest e.s. of (\ref{ap1}) with the initial function $u_0(x+\Delta x)$. This implies the uniform estimate
$$
\int_{\R^n}|u_r(t_0,x+\Delta x)-u_r(t_0,x)|dx\le\int_{\R^n}|u_0(x+\Delta x)-u_0(x)|dx \quad \forall t_0>0.
$$
It follows from this estimate that
\begin{equation}\label{estx}
\int_{\R^n}|u_r(t_0,x+\Delta x)-u_r(t_0,x)|dx\le\omega^x(|\Delta x|),
\end{equation}
where $\omega^x(h)=\sup\limits_{|\Delta x|<h}\int_{\R^n}|u_0(x+\Delta x)-u_0(x)|dx$ is the continuity modulus of
$u_0$ in $L^1(\R^n)$. We then proceed as in \cite{Kr} to get a similar estimate for shifts of the time variable.
For the sake of completeness we provide the details. We choose an averaging kernel $\beta(y)\in C_0^1(\R^n)$ with the properties: $\beta(y)\ge 0$, $\supp\beta(y)\subset B_1(0)=\{ y\in\R^n | |y|\le 1 \}$, $\int_{\R^n}\beta(y)dy=1$. For a function $q(x)\in L^\infty(\R^n)$ we consider the corresponding averaged functions
$$q^h(x)=h^{-n}\int q(y)\beta((x-y)/h)dy, \quad h>0,$$
which are the convolutions $q*\beta^h(x)$, where $\beta^h(x)=h^{-n}\beta(x/h)$. It is clear that $q^h(x)\in C^1(\R^n)$ for each $h>0$, $\|q^h\|_\infty\le\|q\|_\infty$, and $q^h\to q$ as $h\to 0$ a.e. in $\R^n$. Moreover, since $\nabla q^h=q*\nabla\beta^h(x)$, we have
\begin{equation}\label{conder}
\|\nabla q^h\|_\infty\le \frac{c}{h}\|q\|_\infty, \quad c=\|\nabla_y\beta\|_1.
\end{equation}
Applying (\ref{ap1}) with $u=u_r$ to the test function $f=(\theta_\nu(t-t_0)-\theta_\nu(t-t_0-\Delta t))p(x)$, where
$t_0,\Delta t>0$, $p=p(x)\in C_0^1(\R^n)$, $\nu\in\N$, and passing to the limit as $\nu\to\infty$, we get
\begin{equation}\label{15}
\int_{\R^n} (u_r(t_0+\Delta t)-u_r(t_0,x))p(x)dx=\int_{(t_0,t_0+\Delta t)\times\R^n} \varphi_r(u_r)\cdot\nabla pdx.
\end{equation}
By Corollary~\ref{cor1} $\|u_r\|_\infty\le M=\|u_0\|_\infty$ for every $r\in\N$. It follows from the coercivity assumption that there is such $R>0$ that $b(-R)<-M$, $b(R)>M$. All the more, $b_r(-R)<b(R)<-M$, $b_r(R)>b(R)>M$ for all $r\in\N$.
This implies that $(b_r)^{-1}([-M,M])\subset (-R,R)$ and therefore for a.e. $(t,x)\in\Pi$
$$|\varphi_r(u_r)|=g((b_r)^{-1}(u_r))\le N\doteq\max_{|v|\le R} |g(v)|.$$
It now follows from (\ref{15}) that
\begin{equation}\label{16}
\left|\int_{\R^n} (u_r(t_0+\Delta t)-u_r(t_0,x))p(x)dx\right|\le N\|\nabla p\|_1\Delta t.
\end{equation}

Further, we make use of the following variant of Kruzhkov's lemma \cite[Lemma~1]{Kr} (for the sake of completeness, we provide it with the proof).

\begin{lemma}\label{lem2} Let $w(x)\in L^1(\R^n)$. Then for each $h>0$
$$
\int_{\R^n} ||w(x)|-w(x)(\sign w)^h(x)|dx\le 2\omega_w(h),
$$
where $\omega_w(h)=\sup\limits_{|\Delta x|<h}\int_{\R^n}|w(x+\Delta x)-w(x)|dx$ is the continuity modulus of
$w$ in $L^1(\R^n)$.
\end{lemma}

\begin{proof}
First, notice that for each $x,y\in\R^n$
\begin{align*}
||w(x)|-w(x)\sign w(y)|=||w(x)|-(w(x)-w(y))\sign w(y)-w(y)\sign w(y)|= \\||w(x)|-|w(y)|-(w(x)-w(y))\sign w(y)|\le \\
||w(x)|-|w(y)||+|w(x)-w(y)|\le 2|w(x)-w(y)|.
\end{align*}
With the help of above inequality we obtain
\begin{align*}
\int_{\R^n} ||w(x)|-w(x)(\sign w)^h(x)|dx= \\ \int_{\R^n} \left|\int_{\R^n} (|w(x)|-w(x)\sign w(x-y))\beta_h(y)dy\right|dx\le \\
\int_{\R^n} \int_{\R^n} ||w(x)|-w(x)\sign w(x-y)|\beta_h(y)dydx\le \\ \int_{\R^n}\int_{\R^n} 2|w(x)-w(x-y)|\beta_h(y)dydx=\\
2\int_{|y|\le h}\left(\int_{\R^n} |w(x)-w(x-y)|dx\right) \beta_h(y)dy\le2\omega_w(h),
\end{align*}
as was to be proved.
\end{proof}

As it readily follows from Lemma~\ref{lem2}, for any $\rho=\rho(x)\in C_0^1(\R^n)$
\begin{align}\label{17}
\left|\int_{\R^n} |w(x)|\rho(x)dx-\int_{\R^n}w(x)\rho(x)(\sign w)^h(x)dx\right|\le \nonumber\\
\int_{\R^n} ||w(x)|-w(x)(\sign w)^h(x)|\rho(x)dx\le 2\|\rho\|_\infty\omega_w(h).
\end{align}
We apply this relation to the function $w(x)=u_r(t_0+\Delta t,x)-u_r(t_0,x)$ for fixed $t_0,\Delta t>0$, $r\in\N$.
In view of estimate (\ref{estx}) for every $\Delta x\in\R^n$, $|\Delta x|<h$,
\begin{align*}
\int_{\R^n}|w(x+\Delta x)-w(x)|dx\le \int_{\R^n}|u_r(t_0,x+\Delta x)-u_r(t_0,x)|dx+ \\ \int_{\R^n}|u_r(t_0+\Delta t,x+\Delta x)-u_r(t_0+\Delta t,x)|dx\le 2\omega_x(h),
\end{align*}
so that $\omega_w(h)\le 2\omega^x(h)$. It follows from (\ref{17}), (\ref{16}), and (\ref{conder}) that
\begin{align}
\int_{\R^n} |w(x)|\rho(x)dx\le\left|\int_{\R^n}w(x)\rho(x)(\sign w)^h(x)dx\right|+4\|\rho\|_\infty\omega^x(h)=\nonumber\\
\left|\int_{\R^n}(u_r(t_0+\Delta t,x)-u_r(t_0,x))\rho(x)(\sign w)^h(x)dx\right|+4\|\rho\|_\infty\omega^x(h)\le \nonumber\\
N\|\nabla(\rho(x)(\sign w)^h(x))\|_1\Delta t+4\|\rho\|_\infty\omega^x(h)\le c_\rho(\Delta t/h+\omega^x(h)),
\end{align}
where $0<h<1$, and $c_\rho$ is a constant depending only on $\rho$. Since the left hand side of this estimate does not depend on $h$, we arrive at the estimate
\begin{equation}\label{estt}
\int_{\R^n} |u_r(t_0+\Delta t,x)-u_r(t_0,x)|\rho(x)dx\le c_\rho\omega^t(\Delta t),
\end{equation}
where $\displaystyle\omega^t(\Delta t)=\inf_{0<h<1}(\Delta t/h+\omega^x(h))$. Taking $h=(\Delta t)^{1/2}$, we find
$\omega^t(\Delta t)\le (\Delta t)^{1/2}+\omega^x((\Delta t)^{1/2})$ for all $\Delta t\in (0,1)$. Thus, $\omega^t(\Delta t)\to 0$ as $\Delta t\to 0$.
Both estimates (\ref{estx}), (\ref{estt}) are uniform with respect to $t_0>0$ and $r\in\N$. By the known compactness criterium they imply pre-compactness of the sequence $u_r$ in $L^1_{loc}(\Pi)$. Therefore, passing to a subsequence, we can assume that $u_r\to u$ as $r\to\infty$ in $L^1_{loc}(\Pi)$.  We conclude that all the requirements of Proposition~\ref{pro3} are satisfied (with the constant sequence $u_{r0}=u_0$), and by this proposition $u=u(t,x)$ is an e.s. of (\ref{1}), (\ref{ini}).

For more general initial functions $u_0(x)\in (c+L^1(\R^n))\cap L^\infty(\R^n)$, where $c\in\R$, one can make the change $\tilde u=u-c$. As is easy to verify, $u$ is an e.s. of (\ref{1}), (\ref{ini}) if and only if $\tilde u$ is an e.s. to the problem
$$
u_t+\div_x\varphi(c+u)=0, \quad u(0,x)=u_0(x)-c,
$$
corresponding to the parametrization $u=b(v)-c$, $\bar\varphi(c+u)\ni g(v)$.
The existence of such an e.s. has been just shown. This yields the existence of e.s. to the original problem. Thus, we have proved the following result.

\begin{theorem}\label{th1}
For every initial function $u_0\in (c+L^1(\R^n))\cap L^\infty(\R^n)$, where $c\in\R$, there exists an e.s. of problem (\ref{1}), (\ref{ini}).
\end{theorem}

Concerning the uniqueness, it may fail even if $n=1$ and $u_0\in L^1(\R)\cap L^\infty(\R)$.

\begin{example}\label{ex1}
We will study the problem
\begin{equation}\label{18}
u_t+H(u)_x=0,  \quad u(0,x)=u_0(x)\doteq\frac{1}{1+x^2},
\end{equation}
where $H(u)=\sign^+ u$ is the Heaviside function. The natural solution of this problem is the stationary solution
$u(t,x)\equiv u_0(x)$. To construct other e.s., we choose the appropriate continuous parametrization of the flux
(it corresponds (\ref{ext}) if we set $H(0)=1/2$)
$$u=b(v)=\left\{\begin{array}{lcr} v & , & v<0, \\ 0 & , & 0\le v\le 1, \\ v-1 & , & v>1, \end{array}\right. \quad
\tilde H(u)\ni g(v)=\left\{\begin{array}{lcr} 0 & , & v<0, \\ v & , & 0\le v\le 1, \\ 1 & , & v>1, \end{array}\right.$$
where $\tilde H(u)=H(u)$, $u\not=0$, $\tilde H(0)=[0,1]$.
\end{example}
We are going to find an e.s. of (\ref{18}) in the form
$$
u(t,x)=\left\{\begin{array}{lcr} 1/(1+x^2) & , & x>x(t), \\ 0 & , & x<x(t), \end{array}\right.
$$
where $x(t)\in C^1((\alpha,\beta))$, $0\le\alpha<\beta\le+\infty$; $x'(t)>0$, $\lim\limits_{t\to\alpha+} x(t)=-\infty$, $\lim\limits_{t\to\beta-} x(t)=+\infty$ if $\beta<+\infty$. The corresponding measure valued e.s. $\nu_{t,x}$ is assumed being regular, i.e., it is an e.s. $v=v(t,x)\in L^\infty(\Pi)$ of the conservation law
$b(v)_t+g(v)_x=0$ such that $u=b(v)$. In particular, $v(t,x)=1+1/(1+x^2)$ if $x>x(t)$ and $v(t,x)\in [0,1]$ if $x<x(t)$.
Since in the latter case $v_x=b(v)_t+g(v)_x=0$ in the sense of distributions, we claim that $v$ does not depend on $x$, i.e., $v=v(t)$ in the domain $x<x(t)$. As is easy to realize, both the Rankine-Hugoniot and the Oleinik conditions should be fulfilled on the discontinuity line $x=x(t)$.
They means, respectively, that $x'(t)$ coincides with the slope of the chord connected the points $(b(v-),g(v-))$, $(b(v+),g(v+))$ of the graph of the flux function, and that this graph lies above of the indicated chord then $v$ runs between
$v-=\lim\limits_{x\to x(t)-} v(t,x)=v(t)$ and $v+=\lim\limits_{x\to x(t)+} v(t,x)=1+1/(1+x(t)^2)>v-$.
Notice that the Oleinik condition is automatically satisfied while the Rankine-Hugoniot condition provides the differential equation $x'(t)=(1+x^2)(1-v(t))$. In particular, taking $v(t)\equiv 0$ and solving the above equation, we obtain the discontinuity curve $x=x(t)=\tan(t-t_0)$, $t_0-\pi/2<t<t_0+\pi/2$ with the required properties for all $t_0\ge\pi/2$. Varying $v(t)$, we can construct many other e.s. For example, choosing $v(t)=t^2/(1+t^2)$ and a particular solution $x=-1/t$ of the differential equation $x'(t)=(1+x^2)(1-v(t))=(1+x^2)/(1+t^2)$, we find the e.s. $u=1/(1+x^2)$ if $xt>-1$, $u=0$ if $xt<-1$.
We conclude that an e.s. of (\ref{18}) is not unique. In the case of merely continuous flux vector an e.s. of the problem
(\ref{1}), (\ref{ini}) may also be non-unique but only if $n>1$, see \cite{KrPa1,KrPa2}.

\section{Existence of the largest and the smallest e.s.}

To construct the largest e.s., we choose a strictly decreasing sequence $d_r>d=\esssup u_0(x)$, $r\in\N$, and the corresponding sequence $u_r$ of e.s. of (\ref{1}), (\ref{ini}) with initial functions
$$
u_{0r}(x)=\left\{\begin{array}{lcr} u_0(x) & , & |x|\le r, \\ d_r & , & |x|>r. \end{array}\right.
$$
Since $u_{0r}\in (d_r+L^1(\R^n))\cap L^\infty(\R^n)$ an e.s. $u_r$ actually exists by Theorem~\ref{th1}. Observe that $\forall r\in\N$
$
u_0(x)\le u_{0r+1}(x)\le u_{0r}(x)\le d_r \ \mbox{ a.e. on } \R^n, \mbox{ and }
\lim\limits_{r\to\infty} u_{0r}(x)=u_0(x).
$
Denote $\delta_r=d_r-d_{r+1}>0$. By the maximum principle $u_r\le d_r$ for all $r\in\N$. Therefore,
$$
\{(t,x) | u_{r+1}(t,x)\ge u_r(t,x)\}\subset\{(t,x) | d_{r+1}\ge u_r(t,x)\}=\{(t,x) | d_r-u_r(t,x)\ge \delta_r\}.
$$
By Chebyshev's inequality and Corollary~\ref{cor2} for each $T>0$
\begin{align*}
\meas\{ \ (t,x)\in (0,T)\times\R^n \ | \ u_{r+1}(t,x)\ge u_r(t,x) \ \}\le \\ \meas\{ \ (t,x)\in (0,T)\times\R^n \ | \  d_r-u_r(t,x)\ge \delta_r \ \}\le \\
\frac{1}{\delta_r}\int_{(0,T)\times\R^n}|d_r-u_r|dtdx\le  \frac{T}{\delta_r}\int_{\R^n}|d_r-u_{0r}|dx= \frac{T}{\delta_r}\int_{|x|<r}(d_r-u_0)dx<+\infty.
\end{align*}
We see that the assumption of Proposition~\ref{pro2} regarded to the e.s. $u_{r+1}$ and $u_r$ is satisfied and by this proposition $u_{r+1}\le u_r$ a.e. on $\Pi$. Since
$u_{0r}\ge u_0\ge a\doteq\essinf u_0(x)$ then $u_r\ge a$, by the minimum principle. Hence, the sequence
$$
u_r(t,x)\mathop{\to}_{r\to\infty} u_+(t,x)\doteq\inf_{r>0} u_r(t,x)
$$
a.e. on $\Pi$, as well as in $L^1_{loc}(\Pi)$. By Proposition~\ref{pro3} the limit function $u_+$ is an e.s. of original problem (\ref{1}), (\ref{ini}).
Let us demonstrate that $u_+$ is the largest e.s. of this problem.
For that, we choose an arbitrary e.s. $u=u(t,x)$ of (\ref{1}), (\ref{ini}). By the maximum principle,
$u\le d$. Therefore, for each $r\in\N$
$$
\{ (t,x)\in \Pi_T=(0,T)\times\R^n | u\ge u_r\}\subset\{(t,x)\in \Pi_T | d\ge u_r\}= \{ (t,x)\in \Pi_T | d_r-u_r\ge d_r-d\}$$
and consequently
\begin{align*}
\meas\{(t,x)\in \Pi_T | u\ge u_r\}\le\frac{1}{d_r-d}\int_{\Pi_T}|d_r-u_r|dx\le \frac{T}{d_r-d}\int_{|x|<r}(d_r-u_0)dx<+\infty,
\end{align*}
where we use again Chebyshev's inequality and Corollary~\ref{cor2}.
Hence, the requirement of Proposition~\ref{pro2}, applied to the e.s. $u$ and $u_r$, is satisfied and, by the comparison principle,
the inequality $u_0\le u_{0r}$ implies that $u\le u_r$ a.e. on $\Pi$. In the limit as $r\to\infty$ we conclude that $u\le u_+$ a.e. on $\Pi$. Hence, $u_+$ is the unique largest e.s.
The smallest e.s. $u_-$ can be found as $u_-=-\tilde u_+$, where $\tilde u_+$ is the largest e.s. to the problem (\ref{1-}).

We have established the existence of the largest and the smallest e.s. Let us demonstrate that these e.s. satisfy the stability and monotonicity properties with respect to initial data.

\begin{theorem}\label{th2}
Let $u_{1+},u_{2+}\in L^\infty(\Pi)$ be the largest e.s. of (\ref{1}), (\ref{ini}) with initial functions $u_{10}$, $u_{20}$, respectively. Then for a.e. $t>0$
$$
\int_{\R^n} (u_{1+}(t,x)-u_{2+}(t,x))^+dx\le\int_{\R^n} (u_{10}(x)-u_{20}(x))^+dx.
$$
In particular, if $u_{10}\le u_{20}$ a.e. in $\R^n$ then $u_{1+}\le u_{2+}$ a.e. in $\Pi$.
\end{theorem}

\begin{proof}
We choose a decreasing sequence $d_r>d=\max(\esssup u_{10}(x),\esssup u_{20}(x))$, $r\in\N$, and define the following sequences of initial functions
$$
u_{1r}^0(x)=\left\{\begin{array}{lcr} u_{10}(x) & , & |x|\le r, \\ d_r & , & |x|>r, \end{array}\right. \quad
u_{2r}^0(x)=\left\{\begin{array}{lcr} u_{20}(x) & , & |x|\le r, \\ d_r+1 & , & |x|>r. \end{array}\right.
$$
Let $u_{1r}=u_{1r}(t,x)$, $u_{2r}=u_{2r}(t,x)$ be e.s. of problem (\ref{1}), (\ref{ini}) with initial functions $u_{1r}^0$,
$u_{2r}^0$, respectively. As was demonstrated above, the sequences $u_{1r}$, $u_{2r}$ decrease and converges in $L^1_{loc}(\Pi)$ to the largest e.s. $u_{1+}$, $u_{2+}$, respectively. By the maximum principle $u_r\le d_r$ a.e. in $\Pi$ and therefore for each $T>0$
\begin{align*}
\{ (t,x)\in\Pi_T | u_{1r}(t,x)\ge u_{2r}(t,x) \}\subset\{ (t,x)\in\Pi_T | d_r\ge u_{2r}(t,x) \}\subset \\ \{ (t,x)\in\Pi_T | d_r+1-u_{2r}(t,x)\ge 1 \}.
\end{align*}
By Chebyshev inequality and Corollary~\ref{cor2}
\begin{align*}
\meas\{ (t,x)\in\Pi_T | u_{1r}(t,x)\ge u_{2r}(t,x) \}\le\meas\{ (t,x)\in\Pi_T | d_r+1-u_{2r}(t,x)\ge 1 \}\le \\ \int_{\Pi_T}|d_r+1-u_{2r}(t,x)|dtdx\le T\int_{\R^n}|d_r+1-u_{2r}^0(x)|dx= \\ T\int_{|x|<r}(d_r+1-u_{20}(x))dx<\infty,
\end{align*}
which allows to apply Proposition~\ref{pro2} and conclude that for a.e. $t>0$ and all $r\in\N$
\begin{align*}
\int_{\R^n} (u_{1r}(t,x)-u_{2r}(t,x))^+dx\le \int_{\R^n}(u_{1r}^0(x)-u_{2r}^0(x))^+dx=\\
\int_{|x|<r}(u_{10}(x)-u_{20}(x))^+dx\le \int_{\R^n}(u_{10}(x)-u_{20}(x))^+dx.
\end{align*}
To complete the proof, it remains only to pass to the limit as $r\to\infty$ in above relation with the help of Fatou's lemma.
\end{proof}

\begin{corollary}\label{cor3} With notations of Theorem~\ref{th2} for a.e. $t>0$
$$
\int_{\R^n} |u_{1+}(t,x)-u_{2+}(t,x)|dx\le\int_{\R^n} |u_{10}(x)-u_{20}(x)|dx.
$$
\end{corollary}

\begin{proof}
By Theorem~\ref{th2} we find that for a.e. $t>0$
\begin{align*}
\int_{\R^n} (u_{1+}(t,x)-u_{2+}(t,x))^+dx\le\int_{\R^n} (u_{10}(x)-u_{20}(x))^+dx, \\
\int_{\R^n} (u_{2+}(t,x)-u_{1+}(t,x))^+dx\le\int_{\R^n} (u_{20}(x)-u_{10}(x))^+dx.
\end{align*}
Putting these inequalities together, we derive the desired result.
\end{proof}
The analogues of Theorem~\ref{th2} and Corollary~\ref{cor3} for the smallest e.s. follows from the results for the largest e.s. to the problem (\ref{1-}) after the change $u\to -u$.

Let us return to the problem (\ref{18}) from Example~\ref{ex1} and find the largest and the smallest e.s. explicitly.
First, we demonstrate that the largest e.s. $u_+$ coincides with the stationary solution $u_0=1/(1+x^2)$. Since the e.s. $u_+$ is the largest one, then $u_+\ge u_0$. Further, by Proposition~\ref{pro1}
for a.e. $t>0$
$$
\int_{\R} u_+(t,x)dx=\int_{\R} (u_+(t,x)-0)^+dx\le\int_\R (u_0(x)-0)^+dx=\int_\R u_0(x)dx,
$$
which implies the inequality
$$
\int_{\R} (u_+(t,x)-u_0(x))dx\le 0.
$$
Since $u_+\ge u_0$, we conclude that $u_+=u_0(x)$ a.e. in $\Pi$, as was claimed.

Let us show that the smallest e.s. of (\ref{18}) is given by the expression
$$
u_-(t,x)=\tilde u(t,x)\doteq\left\{\begin{array}{lcr} 1/(1+x^2) & , & x>\tan(t-\pi/2), \\ 0 & , & x<\tan(t-\pi/2),
\end{array}\right.
$$
and we agree that $\tilde u\equiv 0$ for $t\ge\pi$. As was shown in Example~\ref{ex1}, $\tilde u$ is indeed an e.s. of (\ref{18}). Therefore, the smallest e.s. $u_-\le\tilde u$. By the minimum principle we also claim that $u_-\ge 0$. Direct calculation shows that
\begin{equation}\label{e1}
\int \tilde u(t,x)dx=\int_{\tan(t-\pi/2)}^{+\infty}\frac{dx}{1+x^2}=(\pi-t)^+.
\end{equation}
Observe that $(u_-)_t+H(u_-)_x=0$ in $\D'(\Pi)$, where we have to choose $H(0)=0$ because $v=v(t)\equiv 0$ for $x<\tan(t-\pi/2)$, see Example~\ref{ex1}. This easily implies that for a.e. $r>0$
$$
\frac{d}{dt}\int_{-r}^r u_-(t,x)dx=H(u_-(t,-r))-H(u_-(t,r))\ge -1 \ \mbox{ in } \D'(\R),
$$
which, in turn, implies the estimate
$
\int_{-r}^r u_-(t,x)dx\ge\int_{-r}^r u_0(x)dx-t.
$
Passing in this estimate to the limit as $r\to+\infty$, we find that
$
\int u_-(t,x)dx\ge \int u_0(x)dx-t=\pi-t.
$
Taking also into account that $u_-\ge 0$, we see that for a.e. $t>0$
$$
\int u_-(t,x)dx\ge (\pi-t)^+.
$$
Comparing this inequality with (\ref{e1}), we get
$$
\int (\tilde u(t,x)-u_-(t,x))dx\le 0
$$
for a.e. $t>0$.  Since $\tilde u\ge u_-$, this implies the desired identity $u_-=\tilde u(t,x)$.

\medskip
In the end of this section we put the example promised in Introduction, which shows the necessity of the multi-valued extension of the flux.

\begin{example}\label{ex2}
Let $n=1$ and $\chi_0(u)$ be a function that is different from zero only at the zero point, where it equals $1$, i.e. $\chi_0(u)$ is the indicator function of the singleton $\{0\}$. We consider the Riemann problem
$$
u_t+(\chi_0(u))_x=0, \quad u(0,x)=H(x),
$$
where $H(x)$ is the Heaviside function.
Putting the entropy relation
$$
|u-k|_t+[\sign(u-k)(\chi_0(u)-\chi_0(k))]_x\le 0
$$
together with the identities
$$
\pm\left((u-k)_t+(\chi_0(u)-\chi_0(k))_x\right)=0,
$$
we get that for each $k\in\R$
\begin{equation}\label{ex3}
((u-k)^\pm)_t+[\sign^\pm(u-k)(\chi_0(u)-\chi_0(k))]_x\le 0 \ \mbox{ in } \D'(\Pi).
\end{equation}
It follows from this relation that $((u-1)^+)_t\le 0$, $((u+\varepsilon)^-)_t\le 0$ in $\D'(\Pi)$ for each $\varepsilon>0$ and since $0\le u(0,x)\le 1$, we find that $(u-1)^+=(u+\varepsilon)^-=0$, that is, $-\varepsilon\le u\le 1$ a.e. in $\Pi$. In view of arbitrariness of $\varepsilon>0$, we see that $0\le u\le 1$ a.e. in $\Pi$. It again follows from (\ref{ex3}) that  $((u-\varepsilon)^+)_t\le 0$ in $\D'(\Pi)$ for every $\varepsilon>0$. This implies that $(u-\varepsilon)^+\le (H(x)-\varepsilon)^+=0$ a.e. in the quarter-plane $t>0$, $x<0$. Since $\varepsilon>0$ is arbitrary, we conclude that $u(t,x)=0$ in this quarter-plane. Now, we will demonstrate that $u=1$ a.e. in the quarter-plane $t>0$, $x>0$.
For that, we apply the relation $(1-u)_t-\chi_0(u)_x=0$ to the test function $f=p(\min(R+T-t-x,x))h(t)$,
where $T>0$, $R>2$, $p(v)\in C^1(\R)$ is a function with the properties $p'\ge 0$, $p(v)=0$ for $v\le 0$, $p(v)>0$ for $v>0$, $p(v)=1$ for $v\ge 1$; $h(t)\in C_0^1((0,T))$, $h\ge 0$ (notice that $p\equiv 1$ in a neighborhood of a singular line $x=R+T-t-x$, $t<T$, which implies that $f\in C_0^1(\Pi)$). As a result, we get
\begin{align}\label{ex4}
\int_\Pi (1-u)ph'(t)dtdx+\int_{x>R+T-t-x}(-(1-u)+\chi_0(u))p'hdtdx+\nonumber\\ \int_{x<R+T-t-x}(-\chi_0(u))p'hdtdx=0.
\end{align}
Observing that $0\le\chi_0(u)\le 1-u$ for $u=u(t,x)\in [0,1]$, and that $p'=p'(\min(R+T-t-x,x))\ge 0$, we find that
the last two integrals in (\ref{ex4}) are non-positive and therefore for all $h=h(t)\in C_0^1((0,T))$, $h\ge 0$
$$
\int_0^T\left(\int_{\R^n}(1-u)p(\min(R+T-t-x,x))dx\right)h'(t)dt=\int_\Pi (1-u)ph'(t)dtdx\ge 0.
$$
This means that
$$
\frac{d}{dt}\int_{\R^n}(1-u)p(\min(R+T-t-x,x))dx\le 0 \ \mbox{ in } \D'((0,T)).
$$
Taking into account the initial condition, we find that for a.e. $t\in (0,T)$
$$
\int_{\R^n}(1-u)p(\min(R+T-t-x,x))dx\le \int_{\R^n}(1-u_0(x))p(\min(R+T-x,x))dx=0
$$
since $u_0(x)=1$ for $x>0$ while $p(\min(R+T-x,x))=p(x)=0$ for $x\le 0$. In the limit as $R\to+\infty$, this relation implies that for a.e. $t\in (0,T)$
$$
\int_{\R^n}(1-u(t,x))p(x)dx=0.
$$
Since $p(x)>0$ for $x>0$, and $T>0$ is arbitrary, we conclude that $u(t,x)=1$ for a.e. $(t,x)\in\Pi$, $x>0$. We have established that our solution $u=H(x)$. But this function is not even a weak solution of our equation because the Rankine-Hugoniot relation $0=\chi_0(1)=\chi_0(0)=1$ is violated on the shock line $x=0$. Hence, our Riemann problem has no e.s. in the Kruzhkov sense. As we already know, there exists an e.s. of our problem in the sense of Definition~\ref{def1},
corresponding to the multi-valued extension $\tilde\chi_0(0)=[0,1]$ of the flux.
The corresponding continuous parametrization can be given by the functions
$$
u=b(v)=\left\{\begin{array}{lcr} v+1 & , & v<-1, \\ 0 & , & -1\le v\le 1, \\ v-1 & , & v>1, \end{array}\right. \quad
\tilde\chi_0(u)\ni g(v)=\left\{\begin{array}{lcr} 0 & , & |v|>1, \\ 1-|v| & , & |v|\le 1. \end{array}\right.
$$
Let us show that the stationary solution $u=H(x)$ is an e.s. of our problem.
The corresponding e.s. $v=v(t,x)$
of the equation $b(v)_t+g(v)_x=0$ can be chosen regular. For $x>0$ it is uniquely determined by the requirement $b(v)=u=1$ and therefore $v=2$. In the case $x<0$ one can chose  $v\equiv -1$ or $v\equiv 1$ (it is even possible to take measure valued function $\nu_{t,x}(v)=(1-\alpha)\delta(v+1)+\alpha\delta(v-1)$, $\alpha=\alpha(t,x)\in [0,1]$). By the construction both the Rankine-Hogoniot and the Oleinik conditions are satisfied in the shock line $x=0$. Hence $H(x)=b(v)$ is the required e.s.
\end{example}

\section{The case of periodic initial functions}

Let us study the particular case when the initial function $u_0(x)$ is periodic, $u_0(x+e)=u_0(x)$ a.e. in $\R^n$ for all $e\in L$, where $L\subset\R^n$ is a lattice of periods. Without loss of generality we may suppose that $L$ is the standard lattice $\Z^n$.

\begin{theorem}\label{th3}
The largest e.s. $u_+$ and the smallest e.s. $u_-$ of the problem (\ref{1}), (\ref{ini}) are space-periodic and coincide: $u_+=u_-$.
\end{theorem}

\begin{proof}
Let $e\in L$. In view of periodicity of the initial function it is obvious that $u(t,x+e)$ is an e.s. of (\ref{1}), (\ref{ini}) if and only if $u(t,x)$ is an e.s. of the same problem. Therefore, $u_+(t,x+e)$ is the largest e.s.
of (\ref{1}), (\ref{ini}) together with $u_+$. By the uniqueness $u_+(t,x+e)=u_+(t,x)$ a.e. on $\Pi$ for all $e\in L$, that is $u_+$ is a space periodic function. In the same way we prove space periodicity of the minimal e.s. $u_-$.
Let $\nu_{t,x}^\pm(v)$ be measure valued e.s. of (\ref{1'}) corresponding to the e.s. $u_\pm$.
In view of (\ref{weak}), we have
\begin{equation}\label{19}
(u_+-u_-)_t+\div_x\int g(v)d(\nu_{t,x}^+-\nu_{t,x}^-)(v)=0 \ \mbox{ in } \D'(\Pi).
\end{equation}
Let $\alpha(t)\in C_0^1(\R_+)$, $\beta(y)\in C_0^1(\R^n)$, $\displaystyle\int_{\R^n}\beta(y)dy=1$. Applying (\ref{19}) to the test function $k^{-n}\alpha(t)\beta(x/k)$, with $k\in\N$, we arrive at the relation
\begin{equation}\label{20}
k^{-n}\int_\Pi(u_+-u_-)\alpha'(t)\beta(x/k)dtdx+k^{-n-1}\int_\Pi Q\cdot\nabla_y\beta(x/k)\alpha(t)dtdx=0,
\end{equation}
where the vector $Q=Q(t,x)=\int g(v)d(\nu_{t,x}^+-\nu_{t,x}^-)(v)\in L^\infty(\Pi,\R^n)$. We observe that
\begin{align*}
k^{-n-1}\left|\int_\Pi Q\cdot\nabla_y\beta(x/k)\alpha(t)dtdx\right|\le k^{-n-1}\|Q\|_\infty\int_\Pi |\nabla_y\beta(x/k)|\alpha(t)dtdx= \\ k^{-1}\|Q\|_\infty\int_\Pi |\nabla_y\beta(y)|\alpha(t)dtdy=c/k, \quad c=\const.
\end{align*}
Therefore, in the limit as $k\to\infty$ the second integral in (\ref{20}) disappears while (see for example \cite[Lemma~2.1]{PaJDE})
$$
k^{-n}\int_\Pi(u_+-u_-)\alpha'(t)\beta(x/k)dtdx\to\int_{\R_+\times\T^n}(u_+-u_-)(t,x)\alpha'(t)dtdx,
$$
where $\T^n=[0,1)^n$ is the periodicity cell (or, the same, the thorus $\R^n/\Z^n$). Hence, after the passage to the limit we get
$$
\int_{\R_+\times\T^n}(u_+-u_-)(t,x)\alpha'(t)dtdx=0 \quad \forall \alpha(t)\in C_0^1(\R_+).
$$
This identity means that
$$
\frac{d}{dt}\int_{\T^n}(u_+(t,x)-u_-(t,x))dx=0 \ \mbox{ in } \D'(\R_+)
$$
and implies, with the help of initial condition (\ref{ini1}), that for a.e. $t>0$
$$
\int_{\T^n}(u_+(t,x)-u_-(t,x))dx=\int_{\T^n}(u_0(x)-u_0(x))dx=0.
$$
Since $u_+\ge u_-$, we conclude that $u_+=u_-$ a.e. on $\Pi$.
\end{proof}
Since any e.s. of (\ref{1}), (\ref{ini}) is situated between $u_-$ and $u_+$, we deduce the following

\begin{corollary}\label{cor4}
An e.s. of (\ref{1}), (\ref{ini}) is unique and coincides with $u_+$.
\end{corollary}

\section{Weak completeness of e.s.}

In the one-dimensional case $n=1$
we consider a bounded sequence $u_r=u_r(t,x)\in L^\infty(\Pi)$, $r\in\N$, of e.s. of equation (\ref{1}) (without a prescribed initial condition), which are periodic with respect to the spatial variable, $u_r(t,x+1)=u_r(t,x)$ a.e. in $\Pi$.
Without loss of generality we can suppose that this sequence converges weakly-$*$ in $L^\infty(\Pi)$ to a function $u=u(t,x)$. It is clear that this function is $x$-periodic. The main result of this section is the following

\begin{theorem}\label{th4}
The limit function $u(t,x)$ is an e.s. of problem (\ref{1}), (\ref{ini}) with some periodic initial function $u_0(x)$.
\end{theorem}

Applying this theorem to the constant sequence $u_r=u$, we obtain that any e.s. of equation (\ref{1}) admits a strong trace $u_0$ at the initial line $t=0$ in the sense of relation (\ref{ini1}). In the case of continuous flux function Theorem~\ref{th4} was proved in \cite{PanSIMA} and was even extended in \cite{PanSIMA1} to the case of a degenerate parabolic equation $u_t+\varphi(u)_x=A(u)_{xx}$. We underline that the statement of Theorem~\ref{th4} is purely one-dimensional, in the case $n>1$ it is no longer valid, see \cite[Remark~3]{PanSIMA}. To prove Theorem~\ref{th4}, we will follow the scheme of paper \cite{PanSIMA}. First of all we will modify the technical lemma \cite[Lemma~2.3]{PanSIMA}.

\begin{lemma}\label{lem3}
Let $\nu$ be a Borel measure with compact support in $\R$ and $p(v)\in C(\R)$ be such a function that
\begin{equation}\label{21}
\int\sign^+(v-k)(p(v)-p(k))d\nu(v)=0 \quad \forall k\in [a,b],
\end{equation}
where $a<b=\max\supp\nu$. Then $p(v)\equiv\const$ on $[a,b]$.
\end{lemma}

\begin{proof}
We choose values $k_1,k_2\in [a,b]$ such that $p(k_1)=\min\limits_{[a,b]} p(v)$, $p(k_2)=\max\limits_{[a,b]} p(v)$. If $p(k_1)<p(b)$ then $k_1<b$. Taking $k=k_1$, we find that the integrand in (\ref{21}) is not negative and strictly positive in an interval $(b-\delta,b]$, $\delta>0$. Since $b=\max\supp\nu$ then $\nu((b-\delta,b])>0$ and the integral in (\ref{21}) is strictly positive, which contradicts to this condition. Hence, $p(k_1)=p(b)$. Similarly, assuming that $p(k_2)>p(b)$ and taking $k=k_2$ in (\ref{21}), we come to a contradiction. Thus, $p(k_2)=p(k_1)=p(b)$, that is,
$\min\limits_{[a,b]} p(v)=\max\limits_{[a,b]} p(v)$. We conclude that $p(v)\equiv\const$ on $[a,b]$.
\end{proof}

\begin{corollary}\label{cor5}
Suppose that
\begin{equation}\label{21a}
\int\sign^-(v-k)(p(v)-p(k))d\nu(v)=0 \quad \forall k\in [a,b],
\end{equation}
where $a=\min\supp\nu<b$. Then $p(v)\equiv\const$ on $[a,b]$.
\end{corollary}

\begin{proof}
After the change $v\to -v$, $k\to -k$, requirement (\ref{21a}) reduces to the following one: $\forall k\in [-b,-a]$
$$
\int\sign^+(v-k)(p(-v)-p(-k))d\tilde\nu(v)=
-\int\sign^-(-v+k)(p(-v)-p(-k))d\tilde\nu(v)=0,
$$
where $\tilde\nu$ is the push-forward measure $l^*\nu$ under the map $l(v)=-v$. Notice that
$-a=\max\supp\tilde\nu$. By Theorem~\ref{th4} we conclude that $p(-v)\equiv\const$ on $[-b,-a]$, which is equivalent to the desired statement.
\end{proof}

Let $\nu_{t,x}^r$, $r\in\N$, be a measure valued e.s. of (\ref{1'}) corresponding to the e.s. $u_r$. Then the sequence $\nu_{t,x}^r$, $r\in\N$, is bounded and, by Theorem~\ref{thTa}, passing to a subsequence if necessary, we can suppose that this sequence converges weakly as $r\to\infty$ to a bounded measure valued function
$\nu_{t,x}\in\MV(\Pi)$. Since $b^*\nu_{t,x}^r(u)=\delta(u-u_r(t,x))$ then for each $p(u)\in C(\R)$
$$
p(u_r)=\int p(b(v))d\nu_{t,x}^r(v)\mathop{\to}_{r\to\infty}\int p(b(v))d\nu_{t,x}(v).
$$
This relation implies that the push-forward measure $\bar\nu_{t,x}=b^*(\nu_{t,x})(u)$ is the limit measure valued function for the sequence $u_r$ (in the sense of Theorem~\ref{thT}). In particular, this measure valued function is space-periodic, $\bar\nu_{t,x+1}=\bar\nu_{t,x}$ for a.e. $(t,x)\in\Pi$. Notice that, in correspondence with (\ref{pr2}), the weak limit function $u(t,x)=\int ud\bar\nu_{t,x}(u)=\int b(v)d\nu_{t,x}(v)$.

Passing to the limit as $r\to\infty$ in the entropy relation
$$
\int_\Pi \left[\int|b(v)-b(k)|d\nu_{t,x}^r(v)f_t+\int\sign(v-k)(g(v)-g(k))d\nu_{t,x}^r(v)f_x\right]dtdx\ge 0,
$$
$k\in\R$, $f=f(t,x)\in C_0^1(\Pi)$, $f\ge 0$, we obtain the relation
$$
\int_\Pi \left[\int|b(v)-b(k)|d\nu_{t,x}(v)f_t+\int\sign(v-k)(g(v)-g(k))d\nu_{t,x}(v)f_x\right]dtdx\ge 0,
$$
which shows that $\nu_{t,x}$ is an e.s. of (\ref{1'}).

Using compensated compactness arguments, we establish the formulated below one more important property of the limit measure valued e.s. $\nu_{t,x}$. We consider even the more general case of equations
\begin{equation}\label{g1}
\varphi_0(v)_t+\varphi_1(v)_x=0,
\end{equation}
where $\varphi_0(v),\varphi_1(v)$ are arbitrary continuous functions.
A measure valued e.s. $\nu_{t,x}\in\MV(\Pi)$ of this equation is characterized by the usual Kruzhkov entropy relation: for all $k\in\R$
\begin{equation}\label{gen}
\frac{\partial}{\partial t}\int\sign(v-k)(\varphi_0(v)-\varphi_0(k))d\nu_{t,x}(v)+\frac{\partial}{\partial x}\int\sign(v-k)(\varphi_1(v)-\varphi_1(k))d\nu_{t,x}(v)\le 0
\end{equation}
in $\D'(\Pi)$. Taking $k=\pm R$, $R\ge\|\nu_{t,x}\|_\infty$, we derive the identity
\begin{align*}
\frac{\partial}{\partial t}\int (\varphi_0(v)-\varphi_0(k))d\nu_{t,x}(v)+\frac{\partial}{\partial x}\int (\varphi_1(v)-\varphi_1(k))d\nu_{t,x}(v)= \\
\frac{\partial}{\partial t}\int \varphi_0(v)d\nu_{t,x}(v)+\frac{\partial}{\partial x}\int\varphi_1(v)d\nu_{t,x}(v)=0 \ \mbox{ in } \D'(\Pi)
\end{align*}
for all $k\in\R$. Putting this identity multiplied by $\pm 1$ together with (\ref{gen}), we get another (equivalent) form of entropy relation (\ref{gen})
\begin{equation}\label{gen1}
\frac{\partial}{\partial t}\int\psi_{0k}^\pm(v)d\nu_{t,x}(v)+\frac{\partial}{\partial x}\int\psi_{1k}^\pm(v)d\nu_{t,x}(v)\le 0 \ \mbox{ in } \D'(\Pi),
\end{equation}
where
$$
\psi_{ik}^\pm(v)=\sign^\pm(v-k)(\varphi_i(v)-\varphi_i(k)), \quad i=0,1, \ k\in\R.
$$
Denote by $\Co A$ the convex hull of a set $A\subset\R^n$. In the case when $A$ is a compact subset of $\R$,
$\Co A=[\min A,\max A]$.

\begin{proposition}\label{pro4} Let $\nu_{t,x}^r$, $r\in\N$, be a sequence of measure valued e.s. of equation (\ref{g1}) such that for a.e. $(t,x)\in\Pi$ and all $r\in\N$ the function $\varphi_0(v)$ is constant on $\Co\supp\nu_{t,x}^r$ (in particular, this condition is always satisfied when the measure valued functions $\nu_{t,x}^r$ are regular). Suppose that this sequence converges weakly to a measure valued function $\nu_{t,x}$ (in the sense of relation (\ref{pr2a})). Then
for a.e. $(t,x)\in\Pi$ there exists a nonzero vector $(\xi_0,\xi_1)\in\R^2$ such that $\xi_0\varphi_0(v)+\xi_1\varphi_1(v)=\const$ on  $\Co\supp\nu_{t,x}$.
\end{proposition}

\begin{proof}
Since $\nu_{t,x}^r$ are measure valued e.s. of (\ref{g1}) then in view of (\ref{gen1}) for all $k\in\R$ the distributions
$$
\alpha_{kr}^\pm\doteq\frac{\partial}{\partial t}\int\psi_{0k}^\pm(v)d\nu_{t,x}^r(v)+\frac{\partial}{\partial x}\int\psi_{1k}^\pm(v)d\nu_{t,x}^r(v)\le 0 \ \mbox{ in } \D'(\Pi).
$$
By the known representation of nonnegative distributions $\alpha_{kr}^\pm=-\mu_{kr}$, where $\mu_{kr}$ are nonnegative locally finite measures on $\Pi$. We use also that $\alpha_{kr}^+=\alpha_{kr}^-$ because
$$
\alpha_{kr}^+-\alpha_{kr}^-=\frac{\partial}{\partial t}\int \varphi_0(v)d\nu_{t,x}^r(v)+\frac{\partial}{\partial x}\int\varphi_1(v)d\nu_{t,x}^r(v)=0 \ \mbox{ in } \D'(\Pi).
$$
It is clear that $\mu_{kr}=0$ for $|k|>M=\sup\limits_r\|\nu_{t,x}^r\|_\infty$ while for $|k|\le M$
\begin{align*}
<\mu_{kr},f>=\int_{\Pi}\left[\int\psi_{0k}^\pm(v)d\nu_{t,x}^r(v)f_t+\int\psi_{1k}^\pm(v)d\nu_{t,x}^r(v)f_x\right]dtdx\le \\
2\max_{|v|\le M}(|\varphi_0(v)|+|\varphi_1(v)|)\int_{\Pi}\max(|f_t|,|f_x|)dtdx\doteq C_f
\end{align*}
for each $f=f(t,x)\in C_0^1(\Pi)$, $f\ge 0$. Since the constants $C_f$ do not depend on $r$, the sequences of nonnegative measures $\mu_{kr}$, $r\in\N$, are bounded in the space $\mathrm{M}_{loc}(\Pi)$ of locally finite measures in $\Pi$ endowed with the standard locally convex topology.
By the Murat interpolation lemma \cite{Mu1} the sequences of distributions $\alpha_{kr}^\pm$, $r\in\N$ are pre-compact in the Sobolev space $H_{loc}^{-1}(\Pi)$. Recall that this space consists of distributions $l$ on $\Pi$ such that for each $f\in C_0^\infty(\Pi)$ the distribution $fl$ lies in the space $H^{-1}(\R^2)$, which is dual to the Sobolev space $H^1(\R^2)$.
The topology of $H_{loc}^{-1}(\Pi)$ is generated by seminorms $\|lf\|_{H^{-1}}$. We fix $k,l\in\R$ and denote
\begin{align*}
P_{kr}^+=\int\psi_{0k}^+(v)d\nu_{t,x}^r(v), \quad Q_{kr}^+=\int\psi_{1k}^+(v)d\nu_{t,x}^r(v), \\
P_{lr}^-=\int\psi_{0l}^-(v)d\nu_{t,x}^r(v), \quad Q_{lr}^-=\int\psi_{1l}^-(v)d\nu_{t,x}^r(v).
\end{align*}
As we already demonstrated, the sequences
$$
\alpha_{kr}^+=\frac{\partial}{\partial t} P_{kr}^++\frac{\partial}{\partial x} Q_{kr}^+, \quad
\alpha_{lr}^-=\frac{\partial}{\partial t} P_{lr}^-+\frac{\partial}{\partial x} Q_{lr}^-
$$
are precompact in $H_{loc}^{-1}(\Pi)$. By the compensated compactness theory (see \cite{Mu,Ta}), the quadratic functional
$\Phi(\lambda)=\lambda_1\lambda_4-\lambda_2\lambda_3$, $\lambda=(\lambda_1,\lambda_2,\lambda_3,\lambda_4)\in\R^4$, is weakly continuous on the sequence $(P_{kr}^+,Q_{kr}^+,P_{lr}^-,Q_{lr}^-)$.
By the definition of the measure valued limit function $\nu_{t,x}$ we find that as $r\to\infty$
\begin{align*}
P_{kr}^+\rightharpoonup P_k^+\doteq\int\psi_{0k}^+(v)d\nu_{t,x}(v), \quad Q_{kr}^+\rightharpoonup Q_k^+\doteq\int\psi_{1k}^+(v)d\nu_{t,x}(v), \\
P_{lr}^-\rightharpoonup P_l^-\doteq\int\psi_{0l}^-(v)d\nu_{t,x}(v), \quad Q_{lr}^+\rightharpoonup Q_l^-\doteq\int\psi_{1l}^-(v)d\nu_{t,x}(v)
\end{align*}
weakly-$*$ in $L^\infty(\Pi)$. By our assumption the function $\varphi_0(v)$ is constant on the segment $\Co\supp\nu_{t,x}^r$ for all $r\in\N$. Therefore,
$\psi_{0k}^+(v)\equiv P_{kr}^+$, $\psi_{0l}^-(v)\equiv P_{lr}^-$ on this segment. It follows from this observation that
\begin{align*}
P_{kr}^+Q_{lr}^--Q_{kr}^+P_{lr}^-=\int (\psi_{0k}^+(v)\psi_{1l}^-(v)-\psi_{1k}^+(v)\psi_{0l}^-(v))d\nu_{t,x}^r(v)
\mathop{\rightharpoonup}_{r\to\infty} \\ \int (\psi_{0k}^+(v)\psi_{1l}^-(v)-\psi_{1k}^+(v)\psi_{0l}^-(v))d\nu_{t,x}(v)
\ \mbox{ weakly-$*$ in } L^\infty(\Pi).
\end{align*}
On the other hand, this limit equals $P_k^+Q_l^--Q_k^+P_l^-$ in view of the mentioned above weak continuity of the functional $\Phi(\lambda)$. Hence, we arrive at the relation
\begin{align}\label{g2}
\int (\psi_{0k}^+(v)\psi_{1l}^-(v)-\psi_{1k}^+(v)\psi_{0l}^-(v))d\nu_{t,x}(v)=
\int\psi_{0k}^+(v)d\nu_{t,x}(v)\int\psi_{1l}^-(v)d\nu_{t,x}(v)-\nonumber\\ \int\psi_{1k}^+(v)d\nu_{t,x}(v)\int\psi_{0l}^-(v)d\nu_{t,x}(v).
\end{align}
Notice that $\psi_{ik}^+(v)=0$ for $v\le k$ while $\psi_{il}^-(v)=0$ for $v\ge l$, where $i=0,1$. Therefore, the integrand in the left hand side of (\ref{g2}) is identically zero whenever $l\le k$. For all such pairs $(k,l)$ we have
\begin{equation}\label{g3}
\int\psi_{0k}^+(v)d\nu_{t,x}(v)\int\psi_{1l}^-(v)d\nu_{t,x}(v)= \int\psi_{1k}^+(v)d\nu_{t,x}(v)\int\psi_{0l}^-(v)d\nu_{t,x}(v).
\end{equation}
Let $\Omega$ be the set of common Lebesgue points of the functions $(t,x)\to\int p(v)d\nu_{t,x}(v)$, $p(v)\in F$, where
$F\subset C(\R)$ is a countable dense set. Since the set $F$ is countable, $\Omega$ is a set of full measure in $\Pi$. By the density of $F$ any point $(t,x)\in\Omega$ is a Lebesgue points of the functions $\int p(v)d\nu_{t,x}(v)$ for all
$p(v)\in C(\R)$. In particular, for each fixed $(t,x)\in\Omega$ the measure $\nu_{t,x}$ is uniquely determined.
Since identity (\ref{g3}) fulfils a.e. in $\Pi$, it holds at each point of $\Omega$. We fix such a point $(t,x)\in\Omega$ and denote $\nu=\nu_{t,x}$, $[a,b]=\Co\supp\nu$. We have to show that $\xi_0\varphi_0(v)+\xi_1\varphi_1(v)=\const$ on  $[a,b]$ for some $\xi=(\xi_0,\xi_1)\in\R^2$, $\xi\not=0$. If $\varphi_0(v)\equiv\const$ on $[a,b]$, we can take $\xi=(1,0)$, thus completing the proof. So, assume that $\varphi_0(v)$ is not constant on $[a,b]$ and, in particular, that $a<b$. We define a smaller segment $[a_1,b_1]$, where
\begin{align*}
a_1=\max\{ c\in [a,b] \ | \ \varphi_0(v)=\varphi_0(a) \ \forall v\in [a,c] \}, \\
b_1=\min\{ c\in [a,b] \ | \ \varphi_0(v)=\varphi_0(b) \ \forall v\in [c,b] \}.
\end{align*}
If $a_1\ge b_1$ then $\varphi_0(v)\equiv\const$ on $[a,b]$, which contradicts to our assumption. Therefore, $a\le a_1<b_1\le b$ and we can choose such $a_2,b_2\in (a_1,b_1)$ that $a_2<b_2$. Observe that $\varphi_0(v)$ cannot be constant on segments
$[a,a_2]$, $[b_2,b]$ (otherwise, $a_1\ge a_2$, $b_1\le b_2$, respectively). Therefore, there exist such $l_0\in [a,a_2]$, $k_0\in [b_2,b]$ that $\varphi_0(l_0)$, $\varphi_0(k_0)$
are extremum values of $\varphi_0(u)$ on the segments $[a,a_2]$, $[b_2,b]$, which are different from $\varphi_0(a)$, $\varphi_0(b)$, respectively. Then, the functions $\psi_{0k_0}^+(v)$, $\psi_{0l_0}^-(v)$ keep their sign and different from zero in neighborhoods of points $b$, $a$, respectively. This implies that
$$
\int\psi_{0k_0}^+(v)d\nu(v)\not=0, \quad \int\psi_{0l_0}^-(v)d\nu(v)\not=0.
$$
Then, by relation (\ref{g3}) (with $\nu_{t,x}=\nu$)
\begin{equation}\label{g4}
\int\psi_{1l}^-(v)d\nu(v)=c\int\psi_{0l}^-(v)d\nu(v) \quad \forall l\in [a,b_2],
\end{equation}
where
$$
c=\int\psi_{1k_0}^+(v)d\nu(v)/\int\psi_{0k_0}^+(v)d\nu(v).
$$
By relation (\ref{g3}) again
\begin{equation}\label{g4a}
\int\psi_{1k}^+(v)d\nu(v)=c_1\int\psi_{0k}^+(v)d\nu(v) \quad \forall k\in [a_2,b],
\end{equation}
where
$$
c_1=\int\psi_{1l_0}^-(v)d\nu(v)/\int\psi_{0l_0}^-(v)d\nu(v).
$$
Moreover, $c_1=c$ in view of (\ref{g4}). Introducing the function $p(v)=\varphi_1(v)-c\varphi_0(v)$, we can write equalities (\ref{g4}), (\ref{g4a}) in the form
\begin{align*}
\int\sign^-(v-l)(p(v)-p(l))d\nu(v)=0 \quad \forall l\in [a,b_2]; \\
\int\sign^+(v-k)(p(v)-p(k))d\nu(v)=0 \quad \forall k\in [a_2,b].
\end{align*}
By Lemma~\ref{lem3} and its Corollary~\ref{cor5}, we conclude that $p(v)$ is constant on each segment $[a,b_2]$, $[a_2,b]$. Since $a_2<b_2$, these segments intersect and therefore $p(v)=-c\varphi_0(v)+\varphi_1(v)\equiv\const$ on $[a,b]=\Co\supp\nu$, $\nu=\nu_{t,x}$.
This completes the proof.
\end{proof}
Notice, that the sequence $\nu_{t,x}^r$ of measure valued e.s. of equation (\ref{1'}) satisfies the requirements of Proposition~\ref{pro4} and we conclude that for a.e. $(t,x)\in\Pi$ there is a vector $\xi=(\xi_0,\xi_1)\in\R^2$, $\xi\not=0$, such that $\xi_0 b(v)+\xi_1 g(v)\equiv\const$ on $\Co\supp\nu_{t,x}$. In the case of linearly non-degenerate  flux Proposition~\ref{pro4} implies the strong convergence of the sequence $u_r$, even without the periodicity requirement.

\begin{corollary}\label{cor6} Assume that the function $\varphi(u)$ is not affine on nondegenerate intervals. Then the sequence $u_r\to u$ as $r\to\infty$ in $L^1_{loc}(\Pi)$ (strongly), and $u=u(t,x)$ is an e.s. of  (\ref{1}).
\end{corollary}

\begin{proof}
By Proposition~\ref{pro4} for a.e. $(t,x)\in\Pi$ there is $\xi=(\xi_0,\xi_1)\in\R^2$, $\xi\not=0$, such that $\xi_0 b(v)+\xi_1 g(v)\equiv\const$ on $\Co\supp\nu_{t,x}$. Let us show that for such $(t,x)$ the function $b(v)\equiv\const$ on
the segment $\Co\supp\nu_{t,x}$. In fact, assuming the contrary, we realize that the component $\xi_1\not=0$ and consequently $g(v)=cb(v)+\const$ for all $v\in \Co\supp\nu_{t,x}$, where $c=-\xi_0/\xi_1$. This means that $\varphi(u)=cu+\const$ on the interior of the non-degenerate interval $\{u=b(v) | v\in \Co\supp\nu_{t,x} \}$. But this contradicts to our assumption. We conclude that $b(v)$ is constant (equaled $u(t,x)$) on $\Co\supp\nu_{t,x}$. Therefore, the measure valued function $\bar\nu_{t,x}=b^*\nu_{t,x}$ is regular, $\bar\nu_{t,x}(u)=\delta(u-u(t,x))$.
In correspondence with Theorem~\ref{thT} the sequence $u_r$ converges to $u(t,x)$ strongly. Moreover, like in the proof of Proposition~\ref{pro3}, we conclude that the limit function $u=u(t,x)$ is an e.s. of (\ref{1}).
\end{proof}

Below, we prove Theorem~\ref{th4} in the general case.

\subsection{Proof of Theorem~\ref{th4}.}

Let $E$ be the set of full measure in $\R_+$, introduced in the proof of Proposition~\ref{pro1}, consisting of such $t>0$  that $(t,x)$ is a Lebesgue point of $u(t,x)$ for almost all $x\in\R$. We remind that $t\in E$ is a common Lebesgue point of all functions $\int_{\R} u(t,x)\rho(x)dx$, $\rho(x)\in L^1(\R)$. We can choose a sequence $t_m\in E$ such that $t_m\to 0$ as $m\to\infty$, and $u(t_m,x)\rightharpoonup u_0(x)\in L^\infty(\R)$ weakly-$*$ in $L^\infty(\R)$. It is clear that $u_0(x)$ is a periodic function, and that $u(t,x)\rightharpoonup u_0(x)$ as $E\ni t\to 0$. Let $\tilde u=\tilde u(t,x)$ be a unique (by Corollary~\ref{cor4}) e.s. of (\ref{1}), (\ref{ini}) with initial function $u_0$, and $\tilde\nu_{t,x}$ be a corresponding measure valued e.s. of equation (\ref{1'}). We are going to demonstrate that $u=\tilde u$. Clearly, this will complete the proof.
Applying the equalities
\begin{align*}
\frac{\partial}{\partial t} u+\frac{\partial}{\partial x}\int g(v)d\nu_{t,x}(v)=
\frac{\partial}{\partial t}\tilde u+\frac{\partial}{\partial x}\int g(v)d\tilde\nu_{t,x}(v)=0 \ \mbox{ in } \D'(\Pi)
\end{align*}
to the test functions $f=k^{-n}\alpha(t)\beta(x/k)$ and passing to the limit as $k\to\infty$, we derive, like in the proof of Theorem~\ref{th3}, that
$$
\frac{d}{dt}\int_0^1 u(t,x)dx=\frac{d}{dt}\int_0^1 \tilde u(t,x)dx=0 \ \mbox{ in } \D'(\R_+).
$$
This implies that for a.e. $t>0$
\begin{equation}\label{22}
\int_0^1 u(t,x)dx=\int_0^1 \tilde u(t,x)dx=I\doteq\int_0^1 u_0(x)dx,
\end{equation}
where we used the initial condition for e.s. $\tilde u$ and the fact that $\forall t\in E$
$$
\int_0^1 u(t,x)dx=\int_0^1 u(t_m,x)dx\mathop{\to}_{m\to\infty}\int_0^1 u_0(x)dx.
$$
Since in $\D'(\Pi)$
$$
\frac{\partial}{\partial t}(u-\tilde u)+\frac{\partial}{\partial x}\left(\int g(v)d\nu_{t,x}(v)-\int g(v)d\tilde\nu_{t,x}(v)\right)=0,
$$
there exists a Lipschitz function $P=P(t,x)$ (a potential) such that
$$
P_x=u-\tilde u, \quad P_t=\int g(v)d\tilde\nu_{t,x}(v)-\int g(v)d\nu_{t,x}(v) \quad \mbox{ in } \D'(\Pi).
$$
By the Lipschitz condition, this function admits continuous extension on the closure $\bar\Pi$. Since $P$ is defined up to an additive constant, we can assume that $P(0,0)=0$. It is clear that $P_x(t,x)\to P_x(0,x)$ weakly in $\D'(\R)$ as $t\to 0$. Taking into account that
$P_x(t,x)=u(t,x)-\tilde u(t,x)\rightharpoonup 0$ as $t\to 0$, running over a set of full measure, we find that $P_x(0,x)=0$ and therefore
$P(0,x)\equiv P(0,0)=0$. Further, by the spatial periodicity of $u-\tilde u$ and the condition
$$
\int_0^1(u-\tilde u)(t,x)dx=0
$$
(following from (\ref{22})), we find that the function $P(t,x)$ is spatially periodic as well, $P(t,x+1)=P(t,x)$.
Applying the doubling variables method \cite{Pan96} to the pair of measure valued e.s. $\nu_{t,x}$, $\tilde\nu_{t,x}$ of equation (\ref{1'}), we arrive at the relation
\begin{align}\label{23}
\frac{\partial}{\partial t}\iint |b(v)-b(w)|d\nu_{t,x}(v)d\tilde\nu_{t,x}(w)+\nonumber\\
\frac{\partial}{\partial x}\iint \sign(v-w)(g(v)-g(w))d\nu_{t,x}(v)d\tilde\nu_{t,x}(w)\le 0 \ \mbox{ in } \D'(\Pi).
\end{align}
Since $b(w)=\tilde u(t,x)$ on $\supp\tilde\nu_{t,x}$ and $b^*\nu_{t,x}=\bar\nu_{t,x}$, we can simplify the first integral
$$
\iint |b(v)-b(w)|d\nu_{t,x}(v)d\tilde\nu_{t,x}(w)=\int |b(v)-\tilde u(t,x)|d\nu_{t,x}(v)=\int |u-\tilde u(t,x)|d\bar\nu_{t,x}(u).
$$
We will need the following key relation
\begin{align}\label{24}
\iint |b(v)-b(w)|d\nu_{t,x}(v)d\tilde\nu_{t,x}(w)P_t(t,x)+\nonumber\\ \iint \sign(v-w)(g(v)-g(w))d\nu_{t,x}(v)d\tilde\nu_{t,x}(w)P_x(t,x)=0 \ \mbox{ a.e. in } \Pi.
\end{align}
We remind that
\begin{align*}
P_t(t,x)=\int g(w)d\tilde\nu_{t,x}(w)-\int g(v)d\nu_{t,x}(v)=-\iint (g(v)-g(w))d\nu_{t,x}(v)d\tilde\nu_{t,x}(w), \\
P_x(t,x)=u-\tilde u=\iint(b(v)-b(w))d\nu_{t,x}(v)d\tilde\nu_{t,x}(w),
\end{align*}
and (\ref{24}) can be written in the more symmetric form
\begin{align}\label{25}
\iint |b(v)-b(w)|d\nu_{t,x}(v)d\tilde\nu_{t,x}(w)\iint (g(v)-g(w))d\nu_{t,x}(v)d\tilde\nu_{t,x}(w)=\nonumber\\
\iint(b(v)-b(w))d\nu_{t,x}(v)d\tilde\nu_{t,x}(w)\iint \sign(v-w)(g(v)-g(w))d\nu_{t,x}(v)d\tilde\nu_{t,x}(w).
\end{align}
To prove (\ref{25}), we fix $(t,x)\in\Pi$, denote $\nu=\nu_{t,x}$, $\tilde\nu=\tilde\nu_{t,x}$, $[a,b]=\Co\supp\nu$,
$[a_1,b_1]=\Co\supp\tilde\nu$, and consider the following four cases:

(i) $[a,b]\cap [a_1,b_1]=\emptyset$. In this case $\sign(v-w)\equiv s$ is constant on $[a,b]\times[a_1,b_1]$. Therefore,
\begin{align*}
\iint |b(v)-b(w)|d\nu(v)d\tilde\nu(w)=s\iint (b(v)-b(w))d\nu(v)d\tilde\nu(w), \\
\iint \sign(v-w)(g(v)-g(w))d\nu(v)d\tilde\nu(w)=s\iint (g(v)-g(w))d\nu(v)d\tilde\nu(w)
\end{align*}
and (\ref{25}) follows;

(ii) $[a,b]\subset [a_1,b_1]$. Since $b(w)$ is constant on $[a_1,b_1]$, we find
\begin{equation}\label{26}
\iint (b(v)-b(w))d\nu(v)d\tilde\nu(w)=\iint |b(v)-b(w)|d\nu(v)d\tilde\nu(w)=0
\end{equation}
and (\ref{25}) is trivial;

(iii) $[a_1,b_1]\subset [a,b]$. In correspondence with Proposition~\ref{pro4} for some nonzero vector $(\xi_0,\xi_1)$ the function $\xi_0b(v)+\xi_1g(v)=\const$ on $[a,b]$. If $\xi_1=0$ then $b(v)\equiv\const$ on $[a,b]$, which implies (\ref{26}), and (\ref{25}) is trivially satisfied. For $\xi_1\not=0$ we find that $g(v)=cb(v)+\const$ on $[a,b]$, $c=-\xi_0/\xi_1$. Therefore,
\begin{align*}
\iint (g(v)-g(w))d\nu(v)d\tilde\nu(w)=c\iint (b(v)-b(w))d\nu(v)d\tilde\nu(w), \\
\iint \sign(v-w)(g(v)-g(w))d\nu(v)d\tilde\nu(w)=c\iint |b(v)-b(w)|d\nu(v)d\tilde\nu(w),
\end{align*}
and (\ref{25}) follows;

(iv) The remaining case: $a<a_1\le b<b_1$ or $a_1<a\le b_1<b$. We consider only the former subcase $a<a_1\le b<b_1$, the latter subcase is treated similarly. Since $b(w)\equiv b(b_1)$ on $[a_1,b_1]$ while $b(v)\le b(b_1)$ for all $v\in [a,b]$, we find that
\begin{equation}\label{27}
\iint |b(v)-b(w)|d\nu(v)d\tilde\nu(w)=-\iint (b(v)-b(w))d\nu(v)d\tilde\nu(w).
\end{equation}
Besides, if $b(v)\equiv\const$ on $[a,b]$ then $b(v)\equiv\const$ on $[a,b_1]=[a,b]\cup [a_1,b_1]$ and we again arrive at (\ref{26}), which readily implies the desired relation (\ref{25}). Thus, assume that $b(v)$ is not constant on $[a,b]$. In view of (\ref{27}) relation (\ref{25}) will follow from the equality
\begin{equation}\label{28}
\iint \sign(v-w)(g(v)-g(w))d\nu(v)d\tilde\nu(w)=-\iint (g(v)-g(w))d\nu(v)d\tilde\nu(w).
\end{equation}
By Proposition~\ref{pro4} we have $g(v)=cb(v)+\const$ on $[a,b]$, where $c=-\xi_0/\xi_1$ (remark that $\xi_1\not=0$, otherwise $b(v)\equiv\const$ on $[a,b]$, which contradicts our assumption). Therefore,
\begin{align*}
\iint \sign(v-w)(g(v)-g(w))d\nu(v)d\tilde\nu(w)= \\ \iint_{[a,b]\times [a_1,b]}\sign(v-w)(g(v)-g(w))d\nu(v)d\tilde\nu(w)-\\
\iint_{[a,b]\times (b,b_1]}(g(v)-g(w))d\nu(v)d\tilde\nu(w)=\\
c\iint_{[a,b]\times [a_1,b]}|b(v)-b(w)|d\nu(v)d\tilde\nu(w)-\iint_{[a,b]\times (b,b_1]}(g(v)-g(w))d\nu(v)d\tilde\nu(w)=\\
-c\iint_{[a,b]\times [a_1,b]}(b(v)-b(w))d\nu(v)d\tilde\nu(w)-\iint_{[a,b]\times (b,b_1]}(g(v)-g(w))d\nu(v)d\tilde\nu(w),
\end{align*}
where we use that $b(v)-b(w)=b(v)-b(b_1)\le 0$ for $v\in [a,b]$, $w\in [a_1,b]$. On the other hand,
\begin{align*}
\iint (g(v)-g(w))d\nu(v)d\tilde\nu(w)= \\ \iint_{[a,b]\times [a_1,b]}(g(v)-g(w))d\nu(v)d\tilde\nu(w)+
\iint_{[a,b]\times (b,b_1]}(g(v)-g(w))d\nu(v)d\tilde\nu(w)=\\
c\iint_{[a,b]\times [a_1,b]}(b(v)-b(w))d\nu(v)d\tilde\nu(w)+\iint_{[a,b]\times (b,b_1]}(g(v)-g(w))d\nu(v)d\tilde\nu(w),
\end{align*}
and (\ref{28}) follows. This completes the proof of relation (\ref{25}).

Let $\rho(r)=r^2/(1+r^2)$. Then the function $q=\rho(P(t,x))$ is nonnegative and Lipschitz. Moreover, by the chain rule for Sobolev derivatives $q_t=\rho'(P)P_t, q_x=\rho'(P)P_x$. Applying (\ref{23}) to the test function $qf$, where
$f=f(t,x)\in C_0^\infty(\Pi)$, $f\ge 0$, we obtain the relation
\begin{equation}\label{29}
\int_\Pi [BP_t+GP_x]f\rho'(P)dtdx+\int_\Pi [Bf_t+Gf_x]qdtdx\ge 0,
\end{equation}
where we denote
\begin{align*}
B=B(t,x)=\iint |b(v)-b(w)|d\nu_{t,x}(v)d\tilde\nu_{t,x}(w), \\ G=G(t,x)=\iint \sign(v-w)(g(v)-g(w))d\nu_{t,x}(v)d\tilde\nu_{t,x}(w).
\end{align*}
In view of relation (\ref{24}) $BP_t+GP_x=0$ a.e. on $\Pi$ and the first integral in (\ref{29}) disappears. Therefore,
$$
\int_\Pi [Bf_t+Gf_x]qdtdx\ge 0.
$$
Taking in this relation $f=k^{-1}\alpha(t)\beta(x/k)$, where $\alpha(t)\in C_0^1(\R_+)$, $\beta(y)\in C_0^1(\R)$ are nonnegative functions, $\int\beta(y)dy=1$, we arrive at the relation
$$
k^{-1}\int_{\Pi}B(t,x)q(t,x)\alpha'(t)\beta(x/k)dtdx+k^{-2}\int_\Pi G(t,x)q(t,x)\alpha(t)\beta'(x/k)dtdx\ge 0.
$$
In the limit as $k\to\infty$ the second term in this relation disappears while the first one
$$
k^{-1}\int_{\Pi}B(t,x)q(t,x)\alpha'(t)\beta(x/k)dtdx\mathop{\to}_{k\to\infty}\int_{\R^+\times [0,1]}B(t,x)q(t,x)\alpha'(t)dtdx,
$$
where we utilize the $x$-periodicity of $B(t,x)q(t,x)$, which allows to apply \cite[Lemma~2.1]{PaJDE}.
As a result, we get
$$
\int_{\R^+\times [0,1]}B(t,x)q(t,x)\alpha'(t)dtdx \quad \forall \alpha(t)\in C_0^1(\R_+), \alpha(t)\ge 0.
$$
This inequality means that
$$
\frac{d}{dt}\int_0^1 B(t,x)q(t,x)dx\le 0 \ \mbox{ in } \D'(\R_+),
$$
and implies that for $t,\tau\in E$, $t>\tau$,
\begin{equation}\label{30}
0\le\int_0^1 B(t,x)q(t,x)dx\le \int_0^1 B(\tau,x)q(\tau,x)dx,
\end{equation}
where $E\subset\R_+$ is a set of full measure.
Observe that $0\le q(\tau,x)\le |P(\tau,x)|=|P(\tau,x)-P(0,x)|\le L\tau$, where $L$ is a Lipschitz constant of $P$ while the function $B(t,x)$ is bounded. Therefore,
$$
\int_0^1 B(\tau,x)q(\tau,x)dx\mathop{\to}_{E\ni\tau\to 0} 0
$$
and it follows from (\ref{30}) that
$\int_0^1 B(t,x)q(t,x)dx=0$. Since $B,q\ge 0$, we find that $Bq=0$ a.e. on $\Pi$. Let $E\subset\Pi$ be the set where $q=0 \Leftrightarrow P=0$, that is, $E=P^{-1}(0)$. By the known properties of Lipschitz functions, $\nabla P=0$ a.e. on $E$.
In particular, $P_x=u-\tilde u=0$ a.e. in $E$. On the other hand, for $(t,x)\in\Pi\setminus E$ the function $q>0$ and therefore $B=0$ a.e. on this set. Since
$$
B(t,x)=\iint |b(v)-b(w)|d\nu_{t,x}(v)d\tilde\nu_{t,x}(w)=\int |b(v)-\tilde u|d\nu_{t,x}(v),
$$
we find that $b(v)=\tilde u$ on $\supp\nu_{t,x}$. In particular, again $u(t,x)=\int b(v)d\nu_{t,x}(v)=\tilde u(t,x)$.
We conclude that $u=\tilde u$ a.e. in $\Pi$, which completes the proof.

\section*{Acknowledgments}
The research was supported by the Russian Science Foundation, grant 22-21-00344.


\begin{thebibliography}{100}
\bibitem{ABK}
Andreianov~B.\,P., B\'enilan~Ph., Kruzhkov~S.\,N. $L^1$-theory of scalar conservation law with
continuous flux function. J. of Functional Analysis. 171 (2000) 15--33.
\bibitem{BGMS11}
Bul\'{i}\v{c}ek~M., Gwiazda~P., M\'{a}lek~J., \'{S}wierczewska Gwiazda A. On scalar hyperbolic
conservation laws with a discontinuous flux. Math. Models Methods Appl. Sci.
21:1 (2011) 89--113.
\bibitem{BGS13}
Bul\'{i}\v{c}ek~M., Gwiazda~P., \'{S}wierczewska Gwiazda A. Multi-dimensional scalar conservation laws
with fluxes discontinuous in the unknown and the spatial variable, Math. Models Methods Appl.
Sci. 23:3 (2013) 407--439.
\bibitem{GSWZ14}
Gwiazda~P., \'{S}wierczewska Gwiazda A., Wittbold P., Zimmermann A. Multi-dimensional scalar balance laws with
discontinuous flux. J. Funct. Anal. 267:8 (2014) 2846--2883.
\bibitem{DFR}
Dias J.-P., Figueira~M. and Rodrigues J.-F. Solutions to a scalar discontinuous conservation
law in a limit case of phase transitions, J. Math. Fluid Mech. 7 (2005) 153--163.
\bibitem{Di}
DiPerna~R.\,J. Measure-valued solutions to conservation laws. Arch. Ration. Mech. Anal. 88 (1985) 223--270.
\bibitem{Kr}
Kruzhkov~S.\,N. First order quasilinear equations in several independent variables.
Mat. Sbornik 81:2 (1970) 228--255; English transl. in Math. USSR Sb. 10:2 (1970)
217--243.
\bibitem{KrPa1}
Kruzhkov~S.\,N., Panov~E.\,Yu. First-order conservative quasilinear laws with an infinite domain of dependence on the initial data. Soviet Math. Dokl. 42 (1991) 316--321.
\bibitem{KrPa2}
Kruzhkov~S.\,N., Panov~E.\,Yu. Osgood's type conditions for uniqueness of entropy solutions to Cauchy problem for quasilinear conservation laws of the first order. Ann. Univ. Ferrara Sez. VII (N.S.) 40 (1994) 31--54.
\bibitem{Mu}
Murat~F. Compacit\'e par compensation. Ann. Scuola Norm. Sup. Pisa Cl. Sci. (4) 5 (1978) 489--507.
\bibitem{Mu1}
Murat~F. L'injection du c\^{o}ne positif de $H^{-1}$ dans $W^{-1,q}$ est compacte pour tout
$q<2$. J. Math. Pures Appl. (9) 60:3 (1981) 309--322.
\bibitem{Pan95}
Panov~E.\,Yu. On sequences of measure-valued solutions of first-order quasilinear equations.
Sb. Math. 81:1 (1995) 211--227.
\bibitem{Pan96}
Panov~E.\,Yu. On measure-valued solutions of the Cauchy problem for a first-order quasilinear equation. Izv: Math. 60:2 (1996) 335--377.
\bibitem{Pan02}
Panov~E.\,Yu. Maximum and minimum generalized entropy solutions to the Cauchy problem for a first-order quasilinear equation. Sb. Math. 193:5 (2002) 727--743.
\bibitem{PanSIMA}
Panov~E.\,Yu. On weak completeness of the set of entropy solutions to a scalar conservation law. SIAM J. Math. Anal. 2009. 41:1 (2009) 26--36.
\bibitem{PanSIMA1}
Panov~E.\,Yu. On weak completeness of the set of entropy solutions to a degenerate non-linear parabolic equation. SIAM J. Math. Anal. 44:1 (2012) 513--535.
\bibitem{PaJHDE}
Panov~E.\,Yu. On the Cauchy problem for scalar conservation laws in the class of Besicovitch almost periodic functions: Global well-posedness and decay property.
J. Hyperbolic Differ. Equ. 13 (2016) 633--659.
\bibitem{PaJDE}
Panov~E.\,Yu. On the decay property for periodic renormalized solutions to scalar conservation laws. J. Differ. Equ. 260:3 (2016) 2704--2728.
\bibitem{Ta}
Tartar L. Compensated compactness and applications to partial differential equations,
in Research Notes in Mathematics, Nonlinear Analysis, and Mechanics. Heriot-Watt Symposium, Vol. 4 (1979) 136--212.

\end{thebibliography}
\end{document}